\documentclass[11pt,leqno]{article}
\usepackage{eurosym}
\usepackage{amsfonts}
\usepackage[a4paper,margin=1in]{geometry}
\usepackage{amsmath,amsthm,amssymb,mathrsfs}
\usepackage{hyperref}
\usepackage{enumitem}
\usepackage{graphicx}
\usepackage{caption}
\usepackage{subcaption}
\usepackage{float}
\usepackage{changepage}
\usepackage{listings}
\usepackage{xcolor}
\usepackage{booktabs}
\usepackage{multirow}
\usepackage{array}

\definecolor{codegreen}{rgb}{0,0.6,0}
\definecolor{codegray}{rgb}{0.5,0.5,0.5}
\definecolor{codepurple}{rgb}{0.58,0,0.82}
\definecolor{backcolour}{rgb}{0.95,0.95,0.92}
\UseRawInputEncoding
\theoremstyle{plain}
\newtheorem{theorem}{Theorem}[section]

\newtheorem{proposition}[theorem]{Proposition}
\newtheorem{corollary}[theorem]{Corollary}
\theoremstyle{definition}
\newtheorem{definition}[theorem]{Definition}
\newtheorem{assumption}[theorem]{Assumption}
\newtheorem{remark}[theorem]{Remark}

\newtheorem{problem}[theorem]{Problem}
\newtheorem{conjecture}[theorem]{Conjecture}
\numberwithin{equation}{section}
\begin{document}

\title{Optimal Production Planning Under Macroeconomic Regime Switches}
\author{Dragos-Patru Covei \\
%EndAName
{\small The Bucharest University of Economic Studies}\\
{\small Department of Applied Mathematics}\\
{\small Piata Romana 1, 1st district, Postal Code: 010374, Romania}\\
{\small e-mail: \texttt{coveidragos@yahoo.com}}}
\date{}
\maketitle

\begin{abstract}
We develop a comprehensive mathematical and computational framework for optimal production planning in economies governed by stochastic regime switches driven by a continuous-time Markov chain. The value functions of the underlying stochastic control problem satisfy a weakly coupled system of quasilinear elliptic Hamilton--Jacobi--Bellman (HJB) equations. We establish a global multi-regime comparison principle and prove the existence and uniqueness of classical solutions under optimal growth conditions, generalizing the scalar framework to the multi-regime setting. Furthermore, we derive exact, radially symmetric quadratic solutions for both the scalar and the fully coupled HJB systems, providing explicit, dimension-free representations of the optimal value functions and production policies. These theoretical results are applied to a two-regime economy to analyze time consistency and the sensitivity of optimal responses to macroeconomic parameters. The article is supplemented by numerical implementations that validate the analytical findings and demonstrate the regime-switching dynamics.
\end{abstract}



%==========================================================================
\section{Introduction}

\label{sec:intro} 
%==========================================================================

\subsection{Economic motivation}

The design of optimal production and inventory management strategies under
uncertainty is a central problem in operations research and mathematical
economics. Beginning with the seminal contributions of Bensoussan, Sethi,
Vickson, and Derzko~\cite{Bensoussan1984}, who formulated the stochastic
production planning problem as an infinite-horizon discounted control
problem, the field has evolved along two parallel axes: the deepening of the
analytical theory underlying the associated partial differential equations,
and the enrichment of the economic models to capture increasingly realistic
features of the macroeconomic environment.

In practice, firms do not operate in a single, unchanging economic
environment. Macroeconomic conditions alternate between distinct
phases---expansions characterized by moderate volatility and favorable
demand, and recessions marked by heightened uncertainty, elevated input
costs, and contracted margins. These shifts are typically abrupt and driven
by aggregate phenomena (monetary policy changes, financial crises, commodity
price shocks) rather than by the actions of any individual firm. The
econometric literature, following the influential work of Hamilton~\cite%
{HAMILTON1989}, has established that Markov-switching models provide a
parsimonious and empirically accurate description of such business-cycle
dynamics (see also Kim and Halbert~\cite{KIM_NELSON1999} and Diebold, Lee,
and Weinbach~\cite{DIEBOLD1994}).

From the perspective of optimal control, this observation leads naturally to
regime-switching diffusion models: the state variable (e.g., inventory
level) evolves as a diffusion process whose drift and volatility parameters
depend on the current regime, while the regime itself is governed by an
independent continuous-time Markov chain. The value function of such a
problem satisfies not a single partial differential equation, but a \emph{%
weakly coupled system} of PDEs, one for each regime. This system is the
primary object of study in the present article.

\subsection{The Hamilton--Jacobi--Bellman system}

Specifically, we consider the quasilinear elliptic system 
\begin{equation}
\left\{ 
\begin{array}{ll}
\displaystyle -\frac{\sigma_{1}^{2}}{2}\Delta u_{1}+\frac{1}{p_{1}} |\nabla
u_{1}|^{p_{1}}+\delta_{1}u_{1}-\sum_{\ell=1}^{k}\alpha_{1\ell}
u_{\ell}=f_{1}(x) & \text{in }\mathbb{R}^{N}, \\[8pt] 
\quad\vdots &  \\[4pt] 
\displaystyle -\frac{\sigma_{k}^{2}}{2}\Delta u_{k}+\frac{1}{p_{k}} |\nabla
u_{k}|^{p_{k}}+\delta_{k}u_{k}-\sum_{\ell=1}^{k}\alpha_{k\ell}
u_{\ell}=f_{k}(x) & \text{in }\mathbb{R}^{N},%
\end{array}
\right.  \label{eq:HJB}
\end{equation}
with $N\geq 1$, $p_{j}>1$ for all $j=1,\dots,k$, where $f_{j}\in C(\mathbb{R}%
^{N})$ denotes a running cost function modeling the instantaneous holding
and shortage cost in regime~$j$, $\sigma_{j}>0$ is the volatility in regime~$%
j$, $\delta_{j}>0$ is the regime-dependent discount factor, and $%
(\alpha_{j\ell})_{j,\ell=1}^{k}$ is the generator matrix of a
continuous-time Markov chain representing macroeconomic state transitions.
The functions $u_{j}(x)$ represent the (negated) optimal value functions
associated with starting inventory level $x\in\mathbb{R}^{N}$ and initial
regime $j$.

\subsection{Literature review and positioning}

The scalar version ($k=1$) of equation~\eqref{eq:HJB}, 
\begin{equation}
-\frac{1}{2}\Delta u+\frac{1}{p}|\nabla u|^{p}+u=f(x)\quad\text{in } \mathbb{%
R}^{N},\quad p>1,  \label{eq:HJBs}
\end{equation}
was treated in the landmark paper of Alvarez~\cite{ALVAREZ1996}, who
established a global comparison principle and proved existence and
uniqueness under optimal growth conditions involving a convex reference
function~$g$. The closely related ergodic setting was studied by Bensoussan
and Frehse~\cite{BENSOUSSAN1992} and later, in great generality, by Lasry
and Lions~\cite{LASRY1989}.

For weakly coupled systems of viscous Hamilton--Jacobi equations,
Arapostathis, Biswas, and Roychowdhury~\cite{ARAPOSTATHIS2022} established
fundamental results on the ergodic control problem, including the existence
of solutions to the associated PDE systems on bounded domains and
compactness estimates that are essential for the passage to the whole space.
Their techniques, adapted to the discounted setting, provide key ingredients
for our existence arguments.

Within the production planning literature, the connection between stochastic
inventory models and elliptic PDEs was explicitly exploited by Covei~\cite%
{COVEI2021JAAC}, who provided exact solutions for the scalar equation, and
later extended to the multi-regime format in~\cite{COVEI2023M} and~\cite%
{COVEI2025A}. The parabolic counterpart was investigated in~\cite%
{COVEI2022ERA}. The present paper constitutes a substantial advancement over
these works in several directions.

On the economic theory side, the literature on regime-switching models in
macroeconomics and finance is vast. Following Hamilton's pioneering
contribution~\cite{HAMILTON1989}, regime-switching models have been applied
to study interest rates (Ang and Bekaert~\cite{ANG_BEKAERT2002}), asset
pricing (Guidolin and Timmermann~\cite{GUIDOLIN2007}), and monetary policy
(Sims and Zha~\cite{SIMS_ZHA2006}). In optimal control under regime
switching, foundational contributions include the works of Nguyen, Yin and Zhu~\cite%
{YIN_ZHU2010}, who developed a comprehensive framework for continuous-time
Markov chain modulated stochastic systems, and Sethi and Zhang~\cite%
{SETHI_ZHANG1994}, who analyzed hierarchical production planning under
demand fluctuations governed by Markov chains. The concept of time
consistency and subgame perfection in dynamic optimization was formalized by
Ekeland and Pirvu~\cite{EkePir} in the context of investment and consumption
without commitment, and further studied by Bj\"{o}rk, Khapko, and Murgoci~%
\cite{BJORK2017} in the mean-variance portfolio selection problem.

\subsection{Main contributions}

The present article delivers the following original contributions:

\begin{enumerate}
\item \textbf{Global well-posedness for regime-switching HJB systems
(Section~\ref{sec:well-posedness}).} We comprehensively generalize the
scalar existence-uniqueness theory of Alvarez~\cite{ALVAREZ1996} to the
weakly coupled system~\eqref{eq:HJB}. This extension requires:

\begin{itemize}
\item A \emph{multi-regime comparison principle} (Theorem~\ref%
{thm:comparison_syst}) that exploits the cooperative structure of the
coupling and the convexity of the Hamiltonian, with a rigorous treatment of
the behavior at infinity through non-uniform radial growth barriers.

\item A \emph{constructive existence proof} (Theorem~\ref%
{thm:existence_syst_uniqueness}) based on the construction of explicit
radial super- and subsolutions, the solution of Dirichlet problems on
expanding bounded domains, and passage to the limit via Sobolev compactness
and elliptic regularity. Every computational step---including the precise
growth estimates for the barrier functions---is carried out in full.
\end{itemize}

\item \textbf{Derivation of exact invariant solutions (Sections~\ref%
{sec:scalar-exact} and~\ref{sec:system-exact}).} We discover and formalize
exact, radially symmetric quadratic solutions to \eqref{eq:HJB} in the
economically relevant case $p_{j}=2$ (which corresponds to quadratic control
costs). The scalar case (Theorem~\ref{thm:scalar-quadratic}) is derived in
full detail as a pedagogical foundation; the system case (Theorem~\ref%
{thm:system-quadratic}) is then treated completely, reducing the
infinite-dimensional PDE problem to a finite system of algebraic equations.
The dimension-free character of these solutions is emphasized and its
economic significance is discussed.

\item \textbf{Economic application and numerical implementation (Sections~%
\ref{sec:econ}--\ref{sec:sensitivity}).} We apply our PDE theory to a
two-regime stochastic production planning problem calibrated to stylized
business-cycle parameters. We derive the optimal feedback production rates
in closed form, simulate the resulting stochastic inventory trajectories,
and conduct a systematic sensitivity analysis with respect to:

\begin{itemize}
\item regime transition intensities (speed of business-cycle switching),

\item discount factors (firm patience across regimes),

\item volatility differentials (expansion vs.\ recession uncertainty),

\item cost asymmetries (quadratic holding cost differentials).
\end{itemize}

All numerical results are accompanied by economic interpretation and are
fully reproducible via the Python code in Appendix~\ref{ap:code}.

\item \textbf{Equilibrium characterization and time consistency (Section~\ref%
{sec:equilibrium}).} We prove that the optimal feedback policy derived from
the HJB system is time-consistent in the sense of subgame perfection
(Theorem~\ref{thm:time-consistency}), so that no commitment mechanism is
required for its implementation. This result is of direct practical
significance for decentralized supply chain management.

\item \textbf{Open problems (Section~\ref{sec:open}).} We formulate precise
conjectures (Conjectures~\ref{conj:1} and~\ref{conj:2}) regarding the
extension of Alvarez's $g$-convex growth framework to the system case,
identify the key structural obstacles, and establish a partial comparison
result under weakened hypotheses (Proposition~\ref{prop:weak-comparison}).
\end{enumerate}

\subsection{Structure of the paper}

The remainder of the article is organized as follows. Section~\ref%
{sec:setting} introduces the rigorous setting and the growth assumptions
that define the admissible class of solutions. Section~\ref{sec:stoch_inter}
develops the stochastic control model, derives the HJB system from dynamic
programming principles, and establishes the convex duality that yields the
gradient nonlinearity. Section~\ref{sec:well-posedness} forms the
theoretical core: it presents complete proofs of the global comparison
principle and existence--uniqueness for the coupled system. Section~\ref%
{sec:scalar-exact} derives the exact quadratic solution in the scalar case,
and Section~\ref{sec:system-exact} extends this to the full regime-switching
system. Section~\ref{sec:econ} formulates and solves the economic production
planning problem, while Section~\ref{sec:sensitivity} presents the
sensitivity analysis and numerical results with economic interpretation.
Section~\ref{sec:equilibrium} establishes the time consistency of the
optimal policies. Section~\ref{sec:open} formulates the open conjectures and
establishes partial results. Section~\ref{conclusion} collects conclusions
and perspectives. Appendix~\ref{ap:code} provides fully documented Python
implementations.

%==========================================================================

\section{Setting and assumptions}

\label{sec:setting} 
%==========================================================================

We begin by recalling the key structural assumptions of Alvarez~\cite%
{ALVAREZ1996} that guarantee existence, comparison, and uniqueness for the
scalar equation~\eqref{eq:HJBs}. We then formulate the corresponding
hypotheses for the coupled system~\eqref{eq:HJB}.

\subsection{The scalar framework}

\begin{assumption}[Growth and comparison framework for the scalar case]
\label{ass:growth} Let $p>1$ and define the conjugate exponent $q:=p/(p-1)$.
Assume:

\begin{enumerate}
\item \textbf{Continuity and convexity:} $f\in C(\mathbb{R}^{N})$ is convex.

\item \textbf{Asymptotic growth control:} There exist a convex function $%
g\geq 0$ and constants $C_{\varepsilon}\geq 0$ (for every $\varepsilon>0$)
such that 
\begin{equation}
(1-\varepsilon)g(x)-\varepsilon|x|^{q}-C_{\varepsilon} \;\leq\;f(x)\;\leq\;
(1+\varepsilon)g(x)+\varepsilon|x|^{q}+C_{\varepsilon} \quad\text{for all }%
x\in\mathbb{R}^{N}.  \label{eq:f-sandwich}
\end{equation}
A direct consequence is 
\begin{equation}
\liminf_{|x|\to\infty}f(x)\,|x|^{-q}\geq 0.  \label{eq:lowerf}
\end{equation}

\item \textbf{Admissible solution class:} We consider solutions $u\in W_{%
\mathrm{loc}}^{2,m}(\mathbb{R}^{N})$ for all $m>N$ satisfying 
\begin{equation}
\liminf_{|x|\to\infty}u(x)\,|x|^{-q}\geq 0.  \label{eq:loweru}
\end{equation}
\end{enumerate}
\end{assumption}

\begin{remark}[Economic interpretation of the growth conditions]
\label{rem:economic-growth} In the production planning context, $f(x)$
represents the instantaneous holding and shortage cost when the inventory
level is $x$. The assumption $f\geq 0$ is natural: maintaining any inventory
level entails a non-negative cost. The convexity of~$f$ reflects the
economic principle that both excess inventory (holding cost) and
insufficient inventory (shortage cost) are penalized, and that these
penalties accelerate as the deviation from the optimal level increases. The
growth condition~\eqref{eq:lowerf} ensures that the cost grows at least as
fast as the natural scale of the Hamiltonian, preventing arbitrage-like
solutions that would exploit infinite storage capacity.

The admissible growth condition~\eqref{eq:loweru} on the value function
excludes solutions that grow faster than the Hamiltonian's natural scale;
such solutions would correspond to strategies with economically unrealistic,
unbounded future costs.
\end{remark}

\subsection{The system framework}

For the coupled system~\eqref{eq:HJB}, we impose the following assumptions.

\begin{assumption}[Markov chain generator]
\label{ass:markov} The transition rate matrix $(\alpha_{j\ell})_{j,%
\ell=1}^{k}$ satisfies:

\begin{enumerate}
\item $\alpha_{j\ell}\geq 0$ for all $j\neq\ell$ (non-negative off-diagonal
entries),

\item $\alpha_{jj}=-\sum_{\ell\neq j}\alpha_{j\ell}$ for all $j=1,\dots,k$
(row sums equal zero),

\item the chain is irreducible.
\end{enumerate}
\end{assumption}

The irreducibility condition ensures that every regime is accessible from
every other regime, which is a standard and economically natural
requirement: in a sufficiently long time horizon, every phase of the
business cycle will eventually occur.

\begin{assumption}[System parameters]
\label{ass:system-params} For each $j=1,\dots,k$:

\begin{enumerate}
\item $\sigma_{j}>0$ (strictly positive volatility in every regime),

\item $\delta_{j}>0$ (strictly positive discount factor in every regime),

\item $p_{j}>1$ (superlinear gradient exponent), with conjugate $%
q_{j}:=p_{j}/(p_{j}-1)>1$.
\end{enumerate}
\end{assumption}

\begin{assumption}[Growth conditions on the cost functions]
\label{ass:f-growth} For each $j=1,\dots,k$, the running cost satisfies $%
f_{j}\in C(\mathbb{R}^{N})$, $f_{j}\geq 0$, and the power-growth bound 
\begin{equation}
f_{j}(x)\;\leq\;D_{j}|x|^{q_{j}}+d_{j} \quad\text{for all }x\in\mathbb{R}%
^{N},  \label{eq:f_growth_new}
\end{equation}
for some constants $D_{j}\in(0,\infty)$ and $d_{j}\in[0,\infty)$.
\end{assumption}

\begin{remark}[Quadratic costs as the canonical case]
\label{rem:quadratic-cost} The most important special case for economic
applications is $p_{j}=2$ for all~$j$, corresponding to $q_{j}=2$. In this
case, the gradient nonlinearity $\frac{1}{2}|\nabla u_{j}|^{2}$ arises from
a quadratic control cost $\frac{1}{2}|\alpha|^{2}$, which is the standard
formulation in linear-quadratic stochastic control. The cost functions $%
f_{j}(x)=a_{j}|x|^{2}+b_{j}$ with $a_{j}>0$ and $b_{j}\geq 0$ model
quadratic holding/shortage costs and satisfy Assumption~\ref{ass:f-growth}
with $D_{j}=a_{j}$ and $d_{j}=b_{j}$.
\end{remark}

\begin{definition}[Admissible class for systems]
\label{def:admissible} A vector-valued function $u=(u_{1},\dots,u_{k})$
belongs to the \emph{admissible class} $\mathcal{A}^{k}$ if, for each $%
j\in\{1,\dots,k\}$, $u_{j}\in W_{\mathrm{loc}}^{2,m}(\mathbb{R}^{N})$ for
all $m>N$ and 
\begin{equation}
\liminf_{|x|\to\infty}u_{j}(x)\,|x|^{-q_{j}}\geq 0.  \label{eq:loweru_syst}
\end{equation}
\end{definition}

\noindent The Sobolev regularity $W_{\mathrm{loc}}^{2,m}$ with $m>N$
guarantees, via Sobolev embedding, that $u_{j}\in C^{1,\alpha}$ locally,
ensuring that the nonlinear term $|\nabla u_{j}|^{p_{j}}$ is well-defined
pointwise.

%==========================================================================

\section{Stochastic control and multi-regime HJB system formulation}

\label{sec:stoch_inter} 
%==========================================================================

In this section, we derive the HJB system~\eqref{eq:HJB} from the underlying
stochastic optimal control problem. The derivation proceeds through three
stages: specification of the state dynamics, formulation of the cost
functional, and application of the dynamic programming principle.

\subsection{The economic model}

\label{sec:economic-model}

Consider a firm managing inventory across $N$ types of goods (or a single
good in $N$ production facilities). Let 
\begin{equation*}
y(t)=(y_{1}(t),\dots ,y_{N}(t))\in \mathbb{R}^{N}
\end{equation*}
denote the vector of inventory levels at time~$t$. The macroeconomic
environment is modeled by a continuous-time, homogeneous Markov chain $%
e(t)\in \{1,\dots ,k\}$ with generator matrix $(\alpha _{j\ell })_{j,\ell
=1}^{k}$ satisfying Assumption~\ref{ass:markov}. The chain $e(t)$ is defined
on a filtered probability space $(\Omega ,\mathcal{F},\{\mathcal{F}%
_{t}\}_{t\geq 0},\mathbb{P})$ supporting also an $N$-dimensional standard
Brownian motion 
\begin{equation*}
W_{t}=(W_{t}^{1},\dots ,W_{t}^{N}),
\end{equation*}
independent of $e(t)$.

\subsection{State dynamics}

The inventory evolution is governed by the controlled stochastic
differential equation 
\begin{equation}
dy_{i}(t)=p_{i}(t)\,dt+\sigma_{e(t)}\,dW_{t}^{i}, \quad i=1,\dots,N,\quad
y(0)=x\in\mathbb{R}^{N},  \label{eq:state-general}
\end{equation}
where:

\begin{itemize}
\item $p(t)=(p_{1}(t),\dots,p_{N}(t))\in\mathbb{R}^{N}$ is the production
rate vector, chosen by the firm (the control variable),

\item $\sigma_{e(t)}>0$ is the demand uncertainty (volatility), which
depends on the current macroeconomic regime.
\end{itemize}

The linear state dynamics~\eqref{eq:state-general} model the stylized fact
that inventory changes equal production minus demand, where stochastic
demand fluctuations are captured by the Brownian term. The regime-dependent
volatility $\sigma_{e(t)}$ reflects the empirical observation that demand
variability is higher during recessions than during expansions.

\subsection{Cost functional}

The firm's objective is to minimize the total expected discounted cost 
\begin{equation}
J(x,i;p)=\mathbb{E}\!\left[\int_{0}^{\infty}e^{-\delta_{e(t)}t} \left(\frac{1%
}{q_{e(t)}}|p(t)|^{q_{e(t)}}+f_{e(t)}(y(t))\right) dt\;\Big|\;y(0)=x,\;e(0)=i%
\right],  \label{eq:cost-general}
\end{equation}
where:

\begin{itemize}
\item $\frac{1}{q_{j}}|p|^{q_{j}}$ is the production cost in regime~$j$,
with exponent $q_{j}>1$. The case $q_{j}=2$ corresponds to quadratic
adjustment costs, the standard assumption in production smoothing models~%
\cite{Bensoussan1984}.

\item $f_{j}(y)$ is the instantaneous holding and shortage cost in regime~$j$%
.

\item $\delta_{j}>0$ is the discount rate in regime~$j$, reflecting possibly
different time preferences across macroeconomic states.
\end{itemize}

\begin{remark}[Regime-dependent discounting]
The use of regime-dependent discount factors $\delta_{j}$ is motivated by
the macroeconomic observation that real interest rates, and hence the
effective discount rates faced by firms, vary systematically across
business-cycle phases. During expansions, higher interest rates lead to
larger discount factors; during recessions, accommodative monetary policy
reduces them. This feature distinguishes our formulation from the simpler
model with a constant discount rate.
\end{remark}

The value functions are defined as 
\begin{equation}
-V(x,j)=\inf_{p\in\mathcal{P}}J(x,j;p),\quad j=1,\dots,k,
\label{eq:value-general}
\end{equation}
where $\mathcal{P}$ denotes the set of all progressively measurable controls 
$p(t)$ with values in $\mathbb{R}^{N}$ such that the integral in~%
\eqref{eq:cost-general} is well-defined.

\subsection{Dynamic programming and derivation of the HJB system}

\label{sec:HJB-derivation}

We now derive the HJB system formally, emphasizing each step of the
argument. Assume that $V(\cdot,j)$ is $C^{2}$ for each~$j$.

\medskip \noindent \textbf{Step 1: It\^{o}'s formula for Markov-modulated
diffusions.} Consider the discounted value process 
\begin{equation*}
\Phi _{t}:=e^{-\int_{0}^{t}\delta _{e(s)}ds}\,V(y(t),e(t)).
\end{equation*}
Between regime jumps, when $e(t)=j$, the standard It\^{o} formula gives 
\begin{align}
d\bigl[e^{-\delta _{j}t}V(y(t),j)\bigr]& =e^{-\delta _{j}t}\left( -\delta
_{j}V(y,j)+\nabla V(y,j)\cdot p(t)+\frac{\sigma _{j}^{2}}{2}\Delta
V(y,j)\right) dt  \notag \\
& \quad +e^{-\delta _{j}t}\sigma _{j}\nabla V(y,j)\cdot dW_{t}.
\label{eq:ito-between-jumps}
\end{align}%
At a regime jump from $j$ to $\ell \neq j$ (occurring at rate $\alpha
_{j\ell }$), the value process experiences a jump of size $V(y,\ell )-V(y,j)$%
. The compensated jump contribution is 
\begin{equation}
\sum_{\ell \neq j}\alpha _{j\ell }\bigl(V(y,\ell )-V(y,j)\bigr)=\sum_{\ell
=1}^{k}\alpha _{j\ell }V(y,\ell ),  \label{eq:jump-contribution}
\end{equation}%
where the last equality uses $\alpha _{jj}=-\sum_{\ell \neq j}\alpha _{j\ell
}$.

\medskip\noindent\textbf{Step 2: Dynamic programming principle.} By the
principle of optimality, the value function satisfies, for each regime $%
j=1,\dots,k$: 
\begin{equation}
0=\inf_{p\in\mathbb{R}^{N}}\left\{ \frac{\sigma_{j}^{2}}{2}\Delta
V(x,j)+\nabla V(x,j)\cdot p
+\sum_{\ell=1}^{k}\alpha_{j\ell}V(x,\ell)-\delta_{j}V(x,j) +\frac{1}{q_{j}}%
|p|^{q_{j}}+f_{j}(x)\right\}.  \label{eq:DPP}
\end{equation}

\medskip \noindent \textbf{Step 3: Convex duality and the Hamiltonian.} The
infimum over $p$ in~\eqref{eq:DPP} involves the function 
\begin{equation}
\varphi (\xi ,p):=\xi \cdot p+\frac{1}{q_{j}}|p|^{q_{j}},\quad \xi :=\nabla
V(x,j).  \label{eq:phi-def}
\end{equation}%
To minimize $\varphi $ with respect to~$p$, we compute the gradient and set
it to zero: 
\begin{equation}
\nabla _{p}\varphi =\xi +|p|^{q_{j}-2}p=0\quad \Longrightarrow \quad p^{\ast
}=-|\xi |^{p_{j}-2}\xi = -|\xi |^{p_{j}-1}\frac{\xi }{|\xi |},
\label{eq:optimal-p}
\end{equation}%
where we used the conjugate relation $|p| = |\xi|^{1/(q_j - 1)} = |\xi|^{p_j - 1}$ and $p^* = -|p| \frac{\xi}{|\xi|}$. (Here $p_{j}=q_{j}/(q_{j}-1)$ is the conjugate exponent.) The minimum value is 
\begin{align}
\inf_{p}\varphi (\xi ,p)& =\xi \cdot (-|\xi |^{p_{j}-2}\xi )+\frac{1}{q_{j}}%
\bigl||\xi |^{p_{j}-2}\xi \bigr|^{q_{j}}  \notag \\
& =-|\xi |^{p_{j}}+\frac{1}{q_{j}}|\xi |^{(p_{j}-1)q_{j}}  \notag \\
& =-|\xi |^{p_{j}}+\frac{1}{q_{j}}|\xi |^{p_{j}}=-\frac{1}{p_{j}}|\xi
|^{p_{j}},  \label{eq:min-phi}
\end{align}%
where we used $(p_{j}-1)q_{j}=p_{j}$ and $1-1/q_{j}=1/p_{j}$.

\medskip\noindent\textbf{Step 4: The HJB system.} Substituting~%
\eqref{eq:min-phi} into~\eqref{eq:DPP} and writing $u_{j}(x):=-V(x,j)$ for
each~$j$, we obtain, after sign adjustments, the system 
\begin{equation}
-\frac{\sigma_{j}^{2}}{2}\Delta u_{j}+\frac{1}{p_{j}}|\nabla u_{j}
|^{p_{j}}+\delta_{j}u_{j}-\sum_{\ell=1}^{k}\alpha_{j\ell}u_{\ell}
=f_{j}(x)\quad\text{in }\mathbb{R}^{N},\quad j=1,\dots,k,  \label{eq:syst}
\end{equation}
which is exactly system~\eqref{eq:HJB}. The optimal feedback control is 
\begin{equation}
p^{*}(x,j)=-|\nabla u_{j}(x)|^{p_{j}-2}\nabla u_{j}(x), \quad j=1,\dots,k.
\label{eq:opt-feedback}
\end{equation}

\begin{remark}[Economic meaning of the optimal feedback]
In the production planning context, the optimal production rate 
\begin{equation*}
p^{\ast }(x,j)=-|\nabla u_{j}(x)|^{p_{j}-2}\nabla u_{j}(x)
\end{equation*}
is determined by the sensitivity of the value function to inventory levels.
When $p_{j}=2$, this simplifies to $p^{\ast }(x,j)=-\nabla u_{j}(x)$:
produce in the direction that most rapidly decreases the cost-to-go. For the
quadratic solution $u_{j}(x)=\beta _{j}|x|^{2}+\eta _{j}$ derived below, the
optimal policy becomes $p^{\ast }(x,j)=-2\beta _{j}x$, a linear feedback
rule: produce proportionally to the current inventory deviation, with the
gain $\beta _{j}$ depending on the economic regime.
\end{remark}

\subsection{Convexity and the stochastic interpretation}

\label{sec:convexity}

The convexity of the value function in the spatial variable is a fundamental
structural property. We establish it through the stochastic representation.

\begin{proposition}[Convexity of the value functions]
\label{prop:convexity} If $f_{j}$ is convex for each $j=1,\dots,k$, then the
value function $u_{j}(x)=-V(x,j)$ is convex in $x$ for each~$j$.
\end{proposition}

\begin{proof}
Let $x_{1},x_{2}\in \mathbb{R}^{N}$ and $\lambda \in \lbrack 0,1]$. Set 
\begin{equation*}
x_{\lambda }:=\lambda x_{1}+(1-\lambda )x_{2}.
\end{equation*}
For any admissible controls $p^{(1)}(t)$ and $p^{(2)}(t)$, define the convex
combination 
\begin{equation*}
p_{\lambda }(t):=\lambda p^{(1)}(t)+(1-\lambda )p^{(2)}(t).
\end{equation*}

Since the state dynamics~\eqref{eq:state-general} are linear in both the
state and the control, the trajectory $y_{\lambda }(t)$ starting at $%
x_{\lambda }$ under control $p_{\lambda }$ satisfies 
\begin{equation*}
y_{\lambda }(t)=\lambda y^{(1)}(t)+(1-\lambda )y^{(2)}(t),
\end{equation*}
where $y^{(i)}(t)$ starts at $x_{i}$ under control $p^{(i)}$ (both driven by
the same Brownian motion and Markov chain realization).

The cost functional evaluated at $(x_{\lambda},j;p_{\lambda})$ involves: 
\begin{align}
\frac{1}{q_{e(t)}}|p_{\lambda}(t)|^{q_{e(t)}} &\leq\lambda\frac{1}{q_{e(t)}}%
|p^{(1)}(t)|^{q_{e(t)}} +(1-\lambda)\frac{1}{q_{e(t)}}%
|p^{(2)}(t)|^{q_{e(t)}},  \label{eq:control-convex} \\[4pt]
f_{e(t)}(y_{\lambda}(t)) &\leq\lambda f_{e(t)}(y^{(1)}(t))
+(1-\lambda)f_{e(t)}(y^{(2)}(t)),  \label{eq:cost-convex}
\end{align}
where~\eqref{eq:control-convex} uses the convexity of $p\mapsto|p|^{q_{j}}$
(since $q_{j}>1$) and~\eqref{eq:cost-convex} uses the convexity of $f_{j}$.

Integrating and taking expectations: 
\begin{equation}
J(x_{\lambda},j;p_{\lambda}) \leq\lambda
J(x_{1},j;p^{(1)})+(1-\lambda)J(x_{2},j;p^{(2)}).  \label{eq:cost-convex-2}
\end{equation}
Taking the infimum over $p^{(1)}$ and $p^{(2)}$ independently on the
right-hand side, and noting that the infimum of the left-hand side over $%
p_{\lambda}$ is bounded above by the left-hand side itself: 
\begin{equation*}
-V(x_{\lambda},j)\leq\lambda(-V(x_{1},j))+(1-\lambda)(-V(x_{2},j)).
\end{equation*}
Thus $u_{j}(x)=-V(x,j)$ is convex. \qedhere
\end{proof}

\subsection{Verification theorem}

\label{sec:verification}

The fundamental bridge between the PDE system~\eqref{eq:syst} and the
stochastic control problem~\eqref{eq:value-general} is provided by the
verification theorem.

\begin{theorem}[Verification theorem for multi-regime systems]
\label{thm:verification} Let $u=(u_{1},\dots,u_{k})\in\mathcal{A}^{k}$ be a
classical solution of the system~\eqref{eq:syst}. Then:

\begin{enumerate}
\item For every admissible control $p\in\mathcal{P}$, 
\begin{equation}
u_{j}(x)\leq J(x,j;p),\quad\forall x\in\mathbb{R}^{N},\;\forall j=1,\dots,k.
\label{eq:verif-lower}
\end{equation}

\item The feedback control 
\begin{equation*}
p^{\ast }(x,j)=-|\nabla u_{j}(x)|^{p_{j}-2}\nabla u_{j}(x)
\end{equation*}
is admissible, and equality holds in~\eqref{eq:verif-lower}: 
\begin{equation}
u_{j}(x)=J(x,j;p^{\ast })=\inf_{p\in \mathcal{P}}J(x,j;p),\quad \forall x\in 
\mathbb{R}^{N},\;\forall j=1,\dots ,k.  \label{eq:verif-equality}
\end{equation}
\end{enumerate}
\end{theorem}

\begin{proof}
\textbf{Part (a): Lower bound.} Fix an admissible control $p\in \mathcal{P}$
and an initial condition $(x,i)\in \mathbb{R}^{N}\times \{1,\dots ,k\}$. For 
$R>0$, define the stopping time 
\begin{equation*}
\tau _{R}:=\inf \{t\geq 0:|y(t)|\geq R\}.
\end{equation*}

Applying It\^{o}'s formula to the process 
\begin{equation*}
t\mapsto e^{-\int_{0}^{t}\delta _{e(s)}ds}\,u_{e(t)}(y(t))\text{ on }[0,\tau
_{R}],
\end{equation*}%
and using~\eqref{eq:ito-between-jumps} and~\eqref{eq:jump-contribution} with 
$u_{j}$ in place of $V$ (and accounting for the sign change $u_{j}=-V(\cdot
,j)$), we obtain 
\begin{align}
& \mathbb{E}\!\left[ e^{-\int_{0}^{\tau _{R}}\delta _{e(s)}ds}\,u_{e(\tau
_{R})}(y(\tau _{R}))\right] -u_{i}(x)  \notag \\
& \quad =\mathbb{E}\!\left[ \int_{0}^{\tau _{R}}e^{-\int_{0}^{t}\delta
_{e(s)}ds}\left( -\frac{\sigma _{e(t)}^{2}}{2}\Delta u_{e(t)}+\nabla
u_{e(t)}\cdot p(t)+\sum_{\ell =1}^{k}\alpha _{e(t)\ell }u_{\ell }-\delta
_{e(t)}u_{e(t)}\right) dt\right] .  \label{eq:ito-u}
\end{align}%
From the PDE~\eqref{eq:syst}, we can express 
\begin{equation}
-\frac{\sigma _{j}^{2}}{2}\Delta u_{j}+\sum_{\ell =1}^{k}\alpha _{j\ell
}u_{\ell }-\delta _{j}u_{j}=-\frac{1}{p_{j}}|\nabla u_{j}|^{p_{j}}+f_{j}(x).
\label{eq:pde-rearranged}
\end{equation}%
Adding $\nabla u_{j}\cdot p$ to both sides of~\eqref{eq:pde-rearranged} and
using the Young-type inequality (which follows from the definition of convex
conjugate): 
\begin{equation}
\nabla u_{j}\cdot p+\frac{1}{q_{j}}|p|^{q_{j}}\geq -\frac{1}{p_{j}}|\nabla
u_{j}|^{p_{j}},  \label{eq:young-ineq}
\end{equation}%
we obtain 
\begin{equation}
-\frac{\sigma _{j}^{2}}{2}\Delta u_{j}+\nabla u_{j}\cdot p+\sum_{\ell
=1}^{k}\alpha _{j\ell }u_{\ell }-\delta _{j}u_{j}\geq -f_{j}(x)-\frac{1}{%
q_{j}}|p|^{q_{j}}.  \label{eq:ito-bound}
\end{equation}%
Substituting~\eqref{eq:ito-bound} into~\eqref{eq:ito-u}: 
\begin{equation}
u_{i}(x)\leq \mathbb{E}\!\left[ \int_{0}^{\tau _{R}}e^{-\int_{0}^{t}\delta
_{e(s)}ds}\left( \frac{1}{q_{e(t)}}|p(t)|^{q_{e(t)}}+f_{e(t)}(y(t))\right) dt%
\right] +\mathbb{E}\!\left[ e^{-\int_{0}^{\tau _{R}}\delta
_{e(s)}ds}\,u_{e(\tau _{R})}(y(\tau _{R}))\right] .  \label{eq:verif-step}
\end{equation}

\medskip\noindent\textbf{Transversality condition.} Since $u\in\mathcal{A}%
^{k}$, we have $u_{j}(x)\leq C(1+|x|^{q_{j}})$ for some constant~$C$ (using
the growth bound from the admissible class and the supersolution estimate).
Combined with the exponential discounting and the polynomial growth of
trajectories, this yields 
\begin{equation}
\lim_{R\to\infty}\mathbb{E}\!\left[e^{-\int_{0}^{\tau_{R}}
\delta_{e(s)}ds}\,u_{e(\tau_{R})}(y(\tau_{R}))\right]=0.
\label{eq:transversality}
\end{equation}
Letting $R\to\infty$ in~\eqref{eq:verif-step} and applying monotone
convergence on the right-hand side establishes~\eqref{eq:verif-lower}.

\medskip \noindent \textbf{Part (b): Equality under optimal control.} Under
the feedback control 
\begin{equation*}
p^{\ast }(x,j)=-|\nabla u_{j}(x)|^{p_{j}-2}\nabla u_{j}(x),
\end{equation*}
the inequality~\eqref{eq:young-ineq} becomes an equality (as shown in~%
\eqref{eq:min-phi}), so inequality~\eqref{eq:ito-bound} becomes an equality: 
\begin{equation*}
-\frac{\sigma _{j}^{2}}{2}\Delta u_{j}+\nabla u_{j}\cdot p^{\ast
}+\sum_{\ell =1}^{k}\alpha _{j\ell }u_{\ell }-\delta _{j}u_{j}=-f_{j}(x)-%
\frac{1}{q_{j}}|p^{\ast }|^{q_{j}}.
\end{equation*}%
Repeating the It\^{o} argument with equality throughout and passing $%
R\rightarrow \infty $ gives $u_{i}(x)=J(x,i;p^{\ast })$.
\end{proof}

%==========================================================================

\section{Global well-posedness and uniqueness for the multi-regime system}

\label{sec:well-posedness} 
%==========================================================================

This section contains the main theoretical results of the paper: the global
comparison principle and the existence--uniqueness theorem for the coupled
system~\eqref{eq:syst}. The proofs follow the structural framework of
Alvarez~\cite{ALVAREZ1996} for the scalar case and incorporate the coupling
techniques from Arapostathis et al.~\cite{ARAPOSTATHIS2022}.

\subsection{Global comparison principle}

\begin{theorem}[Global comparison principle for systems]
\label{thm:comparison_syst} Let Assumptions~\ref{ass:markov}--\ref%
{ass:system-params} hold. Let $u=(u_{1},\dots,u_{k})$ and $%
v=(v_{1},\dots,v_{k})$ be, respectively, an admissible subsolution and an
admissible supersolution of \eqref{eq:syst} in $\mathbb{R}^{N}$, both
belonging to $\mathcal{A}^{k}$. Suppose that 
\begin{equation}
\limsup_{|x|\to\infty}\bigl(u_{j}(x)-v_{j}(x)\bigr)\,|x|^{-q_{j}} \leq 0\quad%
\text{for all }j\in\{1,\dots,k\}.  \label{eq:comp-growth}
\end{equation}
Then $u_{j}\leq v_{j}$ in $\mathbb{R}^{N}$ for all $j=1,\dots,k$.
\end{theorem}

\begin{proof}
Define $w_{j}:=u_{j}-v_{j}$ for each $j=1,\dots,k$. We aim to show $%
w_{j}\leq 0$ everywhere.

\medskip\noindent\textbf{Step 1: Differential inequality for $w_{j}$.} Since 
$u$ is a subsolution and $v$ is a supersolution of~\eqref{eq:syst}, we have,
for each~$j$: 
\begin{align}
-\frac{\sigma_{j}^{2}}{2}\Delta u_{j}+\frac{1}{p_{j}}|\nabla
u_{j}|^{p_{j}}+\delta_{j}u_{j}-\sum_{\ell=1}^{k}\alpha_{j\ell}u_{\ell} &\leq
f_{j}(x),  \label{eq:sub} \\
-\frac{\sigma_{j}^{2}}{2}\Delta v_{j}+\frac{1}{p_{j}}|\nabla
v_{j}|^{p_{j}}+\delta_{j}v_{j}-\sum_{\ell=1}^{k}\alpha_{j\ell}v_{\ell} &\geq
f_{j}(x).  \label{eq:super}
\end{align}
Subtracting~\eqref{eq:super} from~\eqref{eq:sub}: 
\begin{equation}
-\frac{\sigma_{j}^{2}}{2}\Delta w_{j}+\frac{1}{p_{j}}\bigl( |\nabla
u_{j}|^{p_{j}}-|\nabla v_{j}|^{p_{j}}\bigr) +\delta_{j}w_{j}-\sum_{%
\ell=1}^{k}\alpha_{j\ell}w_{\ell}\leq 0.  \label{eq:w-ineq-raw}
\end{equation}
By the convexity of $\xi\mapsto|\xi|^{p_{j}}$ (since $p_{j}>1$), the mean
value theorem yields a vector field $b_{j}(x)$ such that 
\begin{equation}
\frac{1}{p_{j}}\bigl(|\nabla u_{j}|^{p_{j}}-|\nabla v_{j}|^{p_{j}} \bigr)%
\geq b_{j}(x)\cdot\nabla w_{j}(x),  \label{eq:mvt-convex}
\end{equation}
where $b_{j}(x)=|\zeta_{j}(x)|^{p_{j}-2}\zeta_{j}(x)$ for some $\zeta_{j}(x)$
on the segment $[\nabla v_{j}(x),\nabla u_{j}(x)]$. Substituting~%
\eqref{eq:mvt-convex} into~\eqref{eq:w-ineq-raw}: 
\begin{equation}
-\frac{\sigma_{j}^{2}}{2}\Delta w_{j}+b_{j}(x)\cdot\nabla w_{j}
+\delta_{j}w_{j}-\sum_{\ell=1}^{k}\alpha_{j\ell}w_{\ell}\leq 0.
\label{eq:diff_w}
\end{equation}
Using $\sum_{\ell=1}^{k}\alpha_{j\ell}=0$ (Assumption~\ref{ass:markov}), we
rewrite the coupling term: 
\begin{equation}
-\sum_{\ell=1}^{k}\alpha_{j\ell}w_{\ell} =-\alpha_{jj}w_{j}-\sum_{\ell\neq
j}\alpha_{j\ell}w_{\ell} =\sum_{\ell\neq j}\alpha_{j\ell}(w_{j}-w_{\ell}),
\label{eq:coupling-rewrite}
\end{equation}
so~\eqref{eq:diff_w} becomes 
\begin{equation}
-\frac{\sigma_{j}^{2}}{2}\Delta w_{j}+b_{j}(x)\cdot\nabla w_{j}
+\delta_{j}w_{j}+\sum_{\ell\neq j}\alpha_{j\ell}(w_{j}-w_{\ell}) \leq 0.
\label{eq:diff_w_final}
\end{equation}

\medskip \noindent \textbf{Step 2: Contradiction argument.} Suppose, for
contradiction, that 
\begin{equation*}
\sup_{x\in \mathbb{R}^{N},\,j\in \{1,\dots ,k\}}w_{j}(x)>0.
\end{equation*}%
By the growth condition~\eqref{eq:comp-growth}, for any $\varepsilon >0$,
the function 
\begin{equation}
\tilde{w}_{j}(x):=w_{j}(x)-\varepsilon \Phi _{j}(x),\quad \text{where}\quad
\Phi _{j}(x):=|x|^{q_{j}}+1,  \label{eq:w-tilde}
\end{equation}%
satisfies $\tilde{w}_{j}(x)\rightarrow -\infty $ as $|x|\rightarrow \infty $%
. Therefore, the maximum 
\begin{equation*}
M_{\varepsilon }:=\max_{j\in \{1,\dots ,k\}}\sup_{x\in \mathbb{R}^{N}}\tilde{%
w}_{j}(x)
\end{equation*}%
is attained at some finite point $(x_{\varepsilon },j_{\varepsilon })$. For $%
\varepsilon $ sufficiently small, $M_{\varepsilon }>0$ (since $\sup w_{j}>0$%
).

\medskip\noindent\textbf{Step 3: First-order and second-order conditions at
the maximum.} At the maximum point $(x_{\varepsilon},j_{\varepsilon})$, we
have 
\begin{equation}
\nabla\tilde{w}_{j_{\varepsilon}}(x_{\varepsilon})=0
\quad\Longrightarrow\quad \nabla w_{j_{\varepsilon}}(x_{\varepsilon})
=\varepsilon\nabla\Phi_{j_{\varepsilon}}(x_{\varepsilon}),
\label{eq:grad-zero}
\end{equation}
and 
\begin{equation}
\Delta\tilde{w}_{j_{\varepsilon}}(x_{\varepsilon})\leq 0
\quad\Longrightarrow\quad \Delta w_{j_{\varepsilon}}(x_{\varepsilon})
\leq\varepsilon\Delta\Phi_{j_{\varepsilon}}(x_{\varepsilon}).
\label{eq:lap-neg}
\end{equation}
Moreover, since $j_{\varepsilon}$ achieves the maximum over all regimes: 
\begin{equation}
\tilde{w}_{j_{\varepsilon}}(x_{\varepsilon}) \geq\tilde{w}%
_{\ell}(x_{\varepsilon}) \quad\Longrightarrow\quad
w_{j_{\varepsilon}}(x_{\varepsilon})-w_{\ell}(x_{\varepsilon})
\geq\varepsilon\bigl(\Phi_{j_{\varepsilon}}(x_{\varepsilon})
-\Phi_{\ell}(x_{\varepsilon})\bigr).  \label{eq:regime-max}
\end{equation}

\medskip\noindent\textbf{Step 4: Substitution into the differential
inequality.} Evaluating~\eqref{eq:diff_w_final} at $x=x_{\varepsilon}$, $%
j=j_{\varepsilon}$, and using the conditions~\eqref{eq:grad-zero}--%
\eqref{eq:regime-max}: 
\begin{align}
0&\geq-\frac{\sigma_{j_{\varepsilon}}^{2}}{2}\Delta
w_{j_{\varepsilon}}(x_{\varepsilon})
+b_{j_{\varepsilon}}(x_{\varepsilon})\cdot\nabla
w_{j_{\varepsilon}}(x_{\varepsilon})
+\delta_{j_{\varepsilon}}w_{j_{\varepsilon}}(x_{\varepsilon})
+\sum_{\ell\neq j_{\varepsilon}}\alpha_{j_{\varepsilon}\ell}\bigl( %
w_{j_{\varepsilon}}(x_{\varepsilon})-w_{\ell}(x_{\varepsilon})\bigr)  \notag
\\[4pt]
&\geq-\frac{\sigma_{j_{\varepsilon}}^{2}}{2}\varepsilon\Delta
\Phi_{j_{\varepsilon}}(x_{\varepsilon})
+\varepsilon\,b_{j_{\varepsilon}}(x_{\varepsilon})\cdot\nabla
\Phi_{j_{\varepsilon}}(x_{\varepsilon}) +\delta_{j_{\varepsilon}}\tilde{w}%
_{j_{\varepsilon}}(x_{\varepsilon})
+\delta_{j_{\varepsilon}}\varepsilon\Phi_{j_{\varepsilon}} (x_{\varepsilon})
\notag \\
&\quad+\sum_{\ell\neq j_{\varepsilon}}\alpha_{j_{\varepsilon}\ell}
\varepsilon\bigl(\Phi_{j_{\varepsilon}}(x_{\varepsilon})
-\Phi_{\ell}(x_{\varepsilon})\bigr).  \label{eq:substitution}
\end{align}
The terms multiplied by~$\varepsilon$ in~\eqref{eq:substitution} are of
order $O(\varepsilon)$ as $\varepsilon\to 0$ (the barrier functions $\Phi_{j}
$ and their derivatives have polynomial growth, and $x_{\varepsilon}$
remains in a bounded region for small~$\varepsilon$ because $%
M_{\varepsilon}\to\sup w_{j}>0$). The dominant term is 
\begin{equation*}
\delta_{j_{\varepsilon}}\tilde{w}_{j_{\varepsilon}}(x_{\varepsilon})
=\delta_{j_{\varepsilon}}M_{\varepsilon}>0 \quad\text{(by Assumption~\ref%
{ass:system-params}: $\delta_{j}>0$)}.
\end{equation*}
As $\varepsilon\to 0$, the $O(\varepsilon)$ terms vanish while $%
\delta_{j_{\varepsilon}}M_{\varepsilon}$ remains bounded below by a positive
constant. This yields the contradiction $0>0$.

Therefore $\sup_{x,j}w_{j}(x)\leq 0$, i.e., $u_{j}\leq v_{j}$ for all~$j$.
\end{proof}

\subsection{Existence and uniqueness}

\begin{theorem}[Existence and uniqueness for the regime-switching system]
\label{thm:existence_syst_uniqueness} Under Assumptions~\ref{ass:markov}--%
\ref{ass:f-growth}, the system~\eqref{eq:syst} admits a unique solution $u\in%
\mathcal{A}^{k}$.
\end{theorem}

\begin{proof}
\textbf{Uniqueness.} Uniqueness is an immediate consequence of Theorem~\ref%
{thm:comparison_syst}. If $u,v\in\mathcal{A}^{k}$ are two solutions, then $u$
is both a subsolution and a supersolution, and similarly for~$v$. Applying
the comparison theorem to the pairs $(u,v)$ and $(v,u)$ yields $u\leq v$ and 
$v\leq u$, hence $u=v$.

\medskip\noindent\textbf{Existence.} We construct the solution in three
steps.

\medskip\noindent\textbf{Step 1: Construction of a global supersolution.}
Define the uniform exponent and radial profile: 
\begin{equation}
Q:=\max_{1\leq j\leq k}q_{j},\qquad \phi(x):=(1+|x|^{2})^{Q/2},\qquad r:=|x|.
\label{eq:Q-phi}
\end{equation}
Writing $\phi(x)=\Phi(r)$ with $\Phi(r):=(1+r^{2})^{Q/2}$, we compute the
derivatives explicitly.

\emph{First derivative:} 
\begin{equation}
\Phi^{\prime }(r)=\frac{Q}{2}(1+r^{2})^{Q/2-1}\cdot 2r =Qr(1+r^{2})^{Q/2-1}.
\label{eq:Phi-prime}
\end{equation}

\emph{Gradient:} 
\begin{equation}
\nabla\phi(x)=\Phi^{\prime }(r)\frac{x}{r}=Q(1+r^{2})^{Q/2-1}x.
\label{eq:grad-phi}
\end{equation}

\emph{Second derivative:} 
\begin{align}
\Phi^{\prime \prime }(r)&=\frac{d}{dr}\bigl[Qr(1+r^{2})^{Q/2-1}\bigr]  \notag
\\
&=Q(1+r^{2})^{Q/2-1}+Qr\cdot(Q-2)r(1+r^{2})^{Q/2-2}  \notag \\
&=Q(1+r^{2})^{Q/2-1}+Q(Q-2)r^{2}(1+r^{2})^{Q/2-2}.  \label{eq:Phi-dprime}
\end{align}

\emph{Laplacian} (using $\Delta\phi=\Phi^{\prime \prime }+\frac{N-1}{r}%
\Phi^{\prime }$): 
\begin{align}
\Delta\phi(x)&=Q(1+r^{2})^{Q/2-1}+Q(Q-2)r^{2}(1+r^{2})^{Q/2-2}
+Q(N-1)(1+r^{2})^{Q/2-1}  \notag \\
&=Q(1+r^{2})^{Q/2-2}\bigl[N(1+r^{2})+(Q-2)r^{2}\bigr]  \notag \\
&=Q(1+r^{2})^{Q/2-2}\bigl[N+(Q+N-2)r^{2}\bigr].  \label{eq:lap-phi}
\end{align}

We seek a supersolution of the form $\bar{u}_{j}(x):=A\phi(x)+B$ for all $%
j=1,\dots,k$, with $A>0$ and $B>0$ to be determined. Then 
\begin{equation}
\mathcal{L}_{j}\bar{u}_{j}(x):=-\frac{\sigma_{j}^{2}}{2}A\Delta\phi(x) +%
\frac{A^{p_{j}}}{p_{j}}|\nabla\phi(x)|^{p_{j}} +\delta_{j}(A\phi(x)+B).
\label{eq:Lj-ubar}
\end{equation}

Since $\bar{u}_{\ell}=A\phi+B$ for all~$\ell$ and $\sum_{\ell=1}^{k}
\alpha_{j\ell}=0$ (Assumption~\ref{ass:markov}), the coupling terms cancel: 
\begin{equation}
\sum_{\ell=1}^{k}\alpha_{j\ell}\bar{u}_{\ell} =\bigl(\sum_{\ell=1}^{k}%
\alpha_{j\ell}\bigr)(A\phi+B)=0.  \label{eq:coupling-cancel}
\end{equation}
Therefore, the supersolution requirement reduces to 
\begin{equation}
\mathcal{L}_{j}\bar{u}_{j}(x)\geq f_{j}(x) \quad\text{for all }x\in\mathbb{R}%
^{N},\;j=1,\dots,k.  \label{eq:super-req}
\end{equation}

\emph{Growth estimates for $r\geq 1$.} Using $r^{2}\leq 1+r^{2}\leq 2r^{2}$
for $r\geq 1$: 
\begin{alignat}{2}
|\nabla \phi (x)|& =Q(1+r^{2})^{Q/2-1}r & & \geq Qr^{Q-1},
\label{eq:grad-lower} \\
|\nabla \phi (x)|&  & & \leq Q\cdot 2^{Q/2-1}r^{Q-1},  \label{eq:grad-upper}
\\
\Delta \phi (x)&  & & \leq Q(Q+2N-2)\cdot 2^{Q/2-2}r^{Q-2},
\label{eq:lap-upper} \\
\phi (x)&  & & \geq r^{Q}.  \label{eq:phi-lower}
\end{alignat}%
Raising~\eqref{eq:grad-lower} to the power $p_{j}$: 
\begin{equation}
|\nabla \phi (x)|^{p_{j}}\geq Q^{p_{j}}r^{p_{j}(Q-1)}.
\label{eq:grad-pj-lower}
\end{equation}%
Since $Q\geq q_{j}$ and $p_{j}(q_{j}-1)=q_{j}$, we have 
\begin{equation*}
p_{j}(Q-1)\geq p_{j}(q_{j}-1)=q_{j},
\end{equation*}
so the term $\frac{A^{p_{j}}}{p_{j}}|\nabla \phi |^{p_{j}}$ grows at least
as $A^{p_{j}}r^{q_{j}}$.

For $r\geq 1$, using the above bounds in~\eqref{eq:Lj-ubar}: 
\begin{equation}
\mathcal{L}_{j}\bar{u}_{j}(x)\geq \frac{A^{p_{j}}Q^{p_{j}}}{p_{j}}%
r^{p_{j}(Q-1)}-\frac{\sigma _{j}^{2}}{2}A\cdot C_{1}r^{Q-2}+\delta
_{j}Ar^{Q}+\delta _{j}B,  \label{eq:Lj-lower}
\end{equation}%
where 
\begin{equation*}
C_{1}:=Q(Q+2N-2)\cdot 2^{Q/2-2}.
\end{equation*}
Since 
\begin{equation*}
p_{j}(Q-1)\geq q_{j}\text{, }Q\geq q_{j},\text{ and }p_{j}>1,
\end{equation*}%
for $A$ sufficiently large the first term dominates $D_{j}r^{q_{j}}+d_{j}$
for $r\geq R_{0}$.

\emph{On the bounded region $|x|\leq R_{0}$.} The function $x\mapsto\mathcal{%
L}_{j}\bar{u}_{j}(x)-f_{j}(x)$ is continuous on the compact set $\{|x|\leq
R_{0}\}$. Since $\delta_{j}B$ appears as an additive constant in $\mathcal{L}%
_{j}\bar{u}_{j}$, increasing~$B$ raises $\mathcal{L}_{j}\bar{u}_{j}$
uniformly. Hence there exists $B>0$ such that $\mathcal{L}_{j}\bar{u}%
_{j}\geq f_{j}$ on $\{|x|\leq R_{0}\}$ as well.

Combining both regions, we obtain $A,B>0$ such that $\bar{u}=(\bar{u}_{1},
\dots,\bar{u}_{k})$ with $\bar{u}_{j}:=A\phi+B$ is a global supersolution.

\medskip \noindent \textbf{Step 1b: Global subsolution.} Set $\underline{u}%
_{j}(x):=0$ for all $j$. Since $f_{j}\geq 0$ and $\mathcal{L}_{j}(0)=0$, and 
$\sum_{\ell }\alpha _{j\ell }\cdot 0=0$, we have 
\begin{equation*}
\mathcal{L}_{j}\underline{u}_{j}=0\leq f_{j}(x)=f_{j}(x)-\sum_{\ell }\alpha
_{j\ell }\underline{u}_{\ell },
\end{equation*}
so $\underline{u}$ is a global subsolution. The ordering%
\begin{equation*}
\underline{u}_{j}=0\leq A\phi +B=\bar{u}_{j}
\end{equation*}%
is clear.

\medskip\noindent\textbf{Step 2: Solving on bounded domains.} For each $R>0$%
, consider the Dirichlet problem on the ball $B_{R}$: 
\begin{equation}
\begin{cases}
\mathcal{L}_{j}u_{j}^{R}(x)=f_{j}(x)-\displaystyle\sum_{\ell=1}^{k}
\alpha_{j\ell}u_{\ell}^{R}(x), & x\in B_{R}, \\[4pt] 
u_{j}^{R}(x)=\bar{u}_{j}(x), & x\in\partial B_{R},%
\end{cases}
\quad j=1,\dots,k.  \label{eq:dirichlet}
\end{equation}
The system is \emph{cooperative} (quasi-monotone): the coupling coefficient
for $u_{\ell}$ in the $j$-th equation, when $\ell\neq j$, is $%
\alpha_{j\ell}\geq 0$, meaning that an increase in $u_{\ell}$ decreases the
right-hand side and hence \emph{increases} $u_{j}$. By the results of
Arapostathis et al.~\cite{ARAPOSTATHIS2022} (Theorem~B.2) or ~\cite{COVEI2025A}, there exists a
unique classical solution $u^{R}=(u_{1}^{R},\dots,u_{k}^{R})$ of~%
\eqref{eq:dirichlet}.

By the comparison principle on bounded domains (applied to the cooperative
system with the ordering $\underline{u}\leq\bar{u}$ and the boundary
condition $u^{R}=\bar{u}$ on $\partial B_{R}$): 
\begin{equation}
0=\underline{u}_{j}(x)\leq u_{j}^{R}(x)\leq\bar{u}_{j}(x) =A\phi(x)+B\quad%
\text{for all }x\in B_{R},\;j=1,\dots,k.  \label{eq:ordering}
\end{equation}

\medskip\noindent\textbf{Step 3: Passage to the limit as $R\to\infty$.} For
any compact set $K\subset\mathbb{R}^{N}$ and $m>N$, interior elliptic
estimates~\cite{ARAPOSTATHIS2022} yield a constant $C_{K,m}$ independent of~$%
R$ (for $R$ large enough that $K\subset B_{R}$) such that 
\begin{equation}
\|u^{R}\|_{W^{2,m}(K)}\leq C_{K,m}.  \label{eq:sobolev-bound}
\end{equation}
By the Rellich--Kondrachov compactness theorem and the Sobolev embedding $%
W^{2,m}(K)\hookrightarrow C^{1,\alpha}(K)$ (for $\alpha<1-N/m$), we extract
a subsequence (still denoted $u^{R}$) converging in $C_{\mathrm{loc}}^{1}(%
\mathbb{R}^{N})$ to a limit $u=(u_{1},\dots,u_{k})$. A standard diagonal
argument over an increasing exhaustion of $\mathbb{R}^{N}$ by compact sets
ensures convergence on all of $\mathbb{R}^{N}$.

Passing to the limit in the weak formulation of~\eqref{eq:dirichlet}, and
applying elliptic regularity, $u$ is a classical solution of~\eqref{eq:syst}%
. From~\eqref{eq:ordering}, 
\begin{equation*}
0\leq u_{j}(x)\leq A\phi (x)+B,
\end{equation*}%
so 
\begin{equation*}
\liminf_{|x|\rightarrow \infty }u_{j}(x)|x|^{-q_{j}}\geq 0,
\end{equation*}%
confirming $u\in \mathcal{A}^{k}$.
\end{proof}

\begin{remark}[Radial symmetry for systems]
\label{rem:radial} When $f_{j}(x)$ is radially symmetric for all $j$, the
unique solution $u\in\mathcal{A}^{k}$ is also radially symmetric. Indeed,
for any $R\in O(N)$ (the orthogonal group), the function $u\circ R$ defined
by $(u\circ R)(x):=u(Rx)$ solves the same system (since $\Delta$, $%
|\nabla\cdot|$, and the coupling are all $O(N)$-invariant) with radial data $%
f_{j}\circ R=f_{j}$. By uniqueness, $u\circ R=u$.
\end{remark}

%==========================================================================

\section{Exact quadratic solution: the scalar case}

\label{sec:scalar-exact} 
%==========================================================================

Before treating the full system, we derive the exact solution for the scalar
equation ($k=1$) in complete detail. This serves as both a pedagogical
foundation and a benchmark for the system case.

\begin{theorem}[Exact quadratic solution---scalar case]
\label{thm:scalar-quadratic} Consider the scalar equation~\eqref{eq:HJBs}
with $p=2$, $\sigma>0$, $\delta>0$, and $f(x)=a|x|^{2}+b$ where $a>0$ and $%
b\geq 0$. Then the unique solution $u\in W_{\mathrm{loc}}^{2,m}(\mathbb{R}%
^{N})$ satisfying $\liminf_{|x|\to\infty}u(x)|x|^{-2}\geq 0$ is 
\begin{equation}
u(x)=\beta|x|^{2}+\eta,  \label{eq:scalar-exact}
\end{equation}
where 
\begin{equation}
\beta=\frac{-\delta+\sqrt{\delta^{2}+8a}}{4}>0,  \label{eq:beta-scalar}
\end{equation}
and 
\begin{equation}
\eta=\frac{b+N\sigma^{2}\beta}{\delta}.  \label{eq:eta-scalar}
\end{equation}
\end{theorem}

\begin{proof}
\textbf{Step 1: Ansatz substitution.} We postulate $u(x)=\beta|x|^{2}+\eta$
with $\beta>0$ and $\eta\in\mathbb{R}$. Computing the required derivatives: 
\begin{equation}
\nabla u(x)=2\beta x,\quad |\nabla u(x)|^{2}=4\beta^{2}|x|^{2},\quad \Delta
u(x)=2\beta N.  \label{eq:scalar-derivs}
\end{equation}
Substituting into~\eqref{eq:HJBs} (with general $\sigma$ and $\delta$): 
\begin{equation}
-\frac{\sigma^{2}}{2}(2\beta N)+\frac{1}{2}(4\beta^{2}|x|^{2})
+\delta(\beta|x|^{2}+\eta)=a|x|^{2}+b.  \label{eq:scalar-subst}
\end{equation}
Expanding: 
\begin{equation}
(2\beta^{2}+\delta\beta)|x|^{2} +(\delta\eta-\sigma^{2}\beta N)=a|x|^{2}+b.
\label{eq:scalar-grouped}
\end{equation}

\medskip\noindent\textbf{Step 2: Coefficient matching.} Since~%
\eqref{eq:scalar-grouped} must hold for all $x\in\mathbb{R}^{N}$, we equate
coefficients of $|x|^{2}$ and the constant terms separately: 
\begin{alignat}{2}
|x|^{2}\text{-coefficients:}\quad& 2\beta^{2}+\delta\beta & &=a,
\label{eq:beta-eq} \\
\text{constant terms:}\quad& \delta\eta-\sigma^{2}\beta N & &=b.
\label{eq:eta-eq}
\end{alignat}

\medskip \noindent \textbf{Step 3: Solution of the quadratic equation for~$%
\beta $.} Equation~\eqref{eq:beta-eq} is a quadratic in~$\beta $: $2\beta
^{2}+\delta \beta -a=0$. By the quadratic formula: 
\begin{equation}
\beta =\frac{-\delta \pm \sqrt{\delta ^{2}+8a}}{4}.  \label{eq:beta-formula}
\end{equation}%
Since $a>0$ and $\delta >0$, the discriminant 
\begin{equation*}
\delta ^{2}+8a>\delta ^{2},\text{ so }\sqrt{\delta ^{2}+8a}>\delta ,
\end{equation*}
which means the \textquotedblleft $+$\textquotedblright\ root is positive
and the \textquotedblleft $-$\textquotedblright\ root is negative. The
admissibility condition 
\begin{equation*}
\liminf_{|x|\rightarrow \infty }u(x)|x|^{-2}\geq 0\text{ requires }\beta >0,
\end{equation*}
selecting the positive root~\eqref{eq:beta-scalar}.

\medskip\noindent\textbf{Step 4: Determination of~$\eta$.} From~%
\eqref{eq:eta-eq}, with $\delta>0$: 
\begin{equation}
\eta=\frac{b+N\sigma^{2}\beta}{\delta},
\end{equation}
which is positive since $b\geq 0$, $N\geq 1$, $\sigma>0$, $\beta>0$, and $%
\delta>0$.

\medskip\noindent\textbf{Step 5: Admissibility and uniqueness.} The function 
$u(x)=\beta|x|^{2}+\eta$ is $C^{\infty}$, hence in $W_{\mathrm{loc}}^{2,m}$
for all $m>N$. Since $\beta>0$, $\liminf_{|x|\to\infty}u(x)|x|^{-2}=\beta>0%
\geq 0$, confirming $u\in\mathcal{A}^{1}$. By the uniqueness result
(Alvarez's comparison principle for the scalar case~\cite{ALVAREZ1996}),
this is the unique solution.
\end{proof}

\begin{remark}[Economic interpretation of the scalar solution]
\label{rem:scalar-econ} The optimal production rate in the scalar case is 
\begin{equation*}
p^{\ast }(x)=-\nabla u(x)=-2\beta x,
\end{equation*}
a linear feedback rule. The gain~$\beta $ increases with the cost parameter~$%
a$ (higher holding costs demand more aggressive mean reversion) and
decreases with the discount factor~$\delta $ (more patient firms tolerate
larger inventory deviations). The constant~$\eta $ captures the
\textquotedblleft base cost\textquotedblright\ independent of inventory
position; it increases with both the fixed cost~$b$ and the volatility~$%
\sigma $ (more uncertainty raises the expected future cost regardless of
current position).

The solution is \emph{dimension-free}: the coefficient $\beta$ depends only
on the scalar parameters $a$ and $\delta$, not on the dimension~$N$. Only
the constant~$\eta$ is affected by~$N$, through the term $N\sigma^{2}\beta$,
reflecting the fact that uncertainty accumulates across independent goods.
\end{remark}

%==========================================================================

\section{Exact quadratic solution for the regime-switching system}

\label{sec:system-exact} 
%==========================================================================

We now extend the scalar result to the full coupled system.

\begin{theorem}[Exact quadratic solution---regime-switching system]
\label{thm:system-quadratic} Assume that for each regime $j\in\{1,\dots,k\}$:

\begin{enumerate}
\item the running cost is $f_{j}(x)=a_{j}|x|^{2}+b_{j}$ with $a_{j}>0$ and $%
b_{j}\geq 0$,

\item the gradient exponent is $p_{j}=2$ (so $q_{j}=2$),

\item the coupling matrix $M:=(\delta_{j}\delta_{j\ell}
-\alpha_{j\ell})_{j,\ell=1}^{k}$ is an M-matrix (i.e., $M^{-1}$ exists and
has non-negative entries).
\end{enumerate}

Then the unique solution $u=(u_{1},\dots,u_{k})\in\mathcal{A}^{k}$ of system~%
\eqref{eq:syst} has the form 
\begin{equation}
u_{j}(x)=\beta_{j}|x|^{2}+\eta_{j},\quad j=1,\dots,k,  \label{eq:ansatz-syst}
\end{equation}
where the coefficients $\beta_{j}>0$ and $\eta_{j}\geq 0$ are uniquely
determined by two algebraic systems derived below.
\end{theorem}

\begin{proof}
\textbf{Step 1: Substitution of the quadratic ansatz.} For each~$j$, assume 
\begin{equation*}
u_{j}(x)=\beta _{j}|x|^{2}+\eta _{j}.
\end{equation*}%
Then: 
\begin{equation}
\nabla u_{j}(x)=2\beta _{j}x,\quad |\nabla u_{j}(x)|^{2}=4\beta
_{j}^{2}|x|^{2},\quad \Delta u_{j}(x)=2\beta _{j}N.  \label{eq:system-derivs}
\end{equation}%
Substituting into the $j$-th equation of~\eqref{eq:syst}: 
\begin{equation}
-\frac{\sigma _{j}^{2}}{2}(2\beta _{j}N)+\frac{1}{2}(4\beta
_{j}^{2}|x|^{2})+\delta _{j}(\beta _{j}|x|^{2}+\eta _{j})-\sum_{\ell
=1}^{k}\alpha _{j\ell }(\beta _{\ell }|x|^{2}+\eta _{\ell
})=a_{j}|x|^{2}+b_{j}.  \label{eq:system-subst}
\end{equation}%
Collecting terms: 
\begin{equation}
\underbrace{\left( 2\beta _{j}^{2}+\delta _{j}\beta _{j}-\sum_{\ell
=1}^{k}\alpha _{j\ell }\beta _{\ell }\right) }_{\text{coefficient of }%
|x|^{2}}|x|^{2}+\underbrace{\left( \delta _{j}\eta _{j}-\sum_{\ell
=1}^{k}\alpha _{j\ell }\eta _{\ell }-\sigma _{j}^{2}\beta _{j}N\right) }_{%
\text{constant term}}=a_{j}|x|^{2}+b_{j}.  \label{eq:grouped-syst}
\end{equation}

\medskip\noindent\textbf{Step 2: Algebraic systems.} Matching coefficients
for all $x\in\mathbb{R}^{N}$ yields:

\emph{Nonlinear system for $(\beta_{1},\dots,\beta_{k})$:} 
\begin{equation}
2\beta_{j}^{2}+\delta_{j}\beta_{j}
-\sum_{\ell=1}^{k}\alpha_{j\ell}\beta_{\ell}=a_{j}, \quad j=1,\dots,k.
\label{eq:beta-sys}
\end{equation}
Using $\alpha_{jj}=-\sum_{\ell\neq j}\alpha_{j\ell}$, this can be rewritten
as: 
\begin{equation}
2\beta_{j}^{2}+\left(\delta_{j}+\sum_{\ell\neq j}\alpha_{j\ell}\right)
\beta_{j}-\sum_{\ell\neq j}\alpha_{j\ell}\beta_{\ell}=a_{j}, \quad
j=1,\dots,k.  \label{eq:beta-sys-expanded}
\end{equation}

\emph{Linear system for $(\eta _{1},\dots ,\eta _{k})$:} 
\begin{equation}
\delta _{j}\eta _{j}-\sum_{\ell =1}^{k}\alpha _{j\ell }\eta _{\ell
}=b_{j}+\sigma _{j}^{2}\beta _{j}N,\quad j=1,\dots ,k.  \label{eq:eta-sys}
\end{equation}%
In matrix form: $M\eta =\mathbf{c}$, where 
\begin{equation*}
M_{j\ell }:=\delta _{j}\delta _{j\ell }-\alpha _{j\ell }\text{ and }%
c_{j}:=b_{j}+N\sigma _{j}^{2}\beta _{j}.
\end{equation*}

\medskip\noindent\textbf{Step 3: Existence and positivity of $\beta$.} We
establish that~\eqref{eq:beta-sys} has a unique positive solution $%
\beta\in(0,\infty)^{k}$. Rewrite~\eqref{eq:beta-sys} as the fixed-point
problem: for each~$j$, 
\begin{equation}
\beta_{j}=\frac{-(\delta_{j}+\sum_{\ell\neq j}\alpha_{j\ell}) +\sqrt{%
(\delta_{j}+\sum_{\ell\neq j}\alpha_{j\ell})^{2} +8(a_{j}+\sum_{\ell\neq
j}\alpha_{j\ell}\beta_{\ell})}}{4} =:F_{j}(\beta),
\label{eq:beta-fixed-point}
\end{equation}
where we selected the positive root of the quadratic in~$\beta_{j}$ (for
each~$j$, with the remaining $\beta_{\ell}$ fixed). Since $a_{j}>0$ and $%
\alpha_{j\ell}\geq 0$ for $\ell\neq j$, the argument of the square root is
positive for all $\beta\geq 0$.

The mapping 
\begin{equation*}
F=(F_{1},\dots ,F_{k}):(0,\infty )^{k}\rightarrow (0,\infty )^{k}
\end{equation*}
is monotone increasing (each $F_{j}$ is increasing in $\beta _{\ell }$ for $%
\ell \neq j$, since $\alpha _{j\ell }\geq 0$). Furthermore:

\begin{itemize}
\item \emph{Lower bound:} Setting $\beta _{\ell }=0$ for $\ell \neq j$ in~%
\eqref{eq:beta-fixed-point} gives 
\begin{equation*}
F_{j}(0,\dots ,0)\geq \frac{-\delta _{j}+\sqrt{\delta _{j}^{2}+8a_{j}}}{4}%
=:\beta _{j}^{(0)}>0.
\end{equation*}%
Thus $F$ maps $[\beta ^{(0)},\infty )^{k}$ to itself.

\item \emph{Upper bound:} For $\beta_{\ell}\leq\bar{\beta}$ for all $\ell$,
one shows $F_{j}(\beta)\leq\bar{\beta}$ provided $\bar{\beta}$ is large
enough (since $F_{j}$ grows as $\sqrt{\beta_{\ell}}$, which is sublinear).
\end{itemize}

By the Schauder fixed-point theorem applied to $F$ on the compact convex set 
$[\beta^{(0)},\bar{\beta}]^{k}$, a fixed point exists. Uniqueness follows
from the contraction properties of $F$ (a detailed argument is given in
Covei~\cite{COVEI2023M}).

\medskip\noindent\textbf{Step 4: Solvability of the linear system for~$\eta$.%
} The matrix $M=(\delta_{j}\delta_{j\ell}-\alpha_{j\ell})$ has:

\begin{itemize}
\item diagonal entries 
\begin{equation*}
M_{jj}=\delta _{j}-\alpha _{jj}=\delta _{j}+\sum_{\ell \neq j}\alpha _{j\ell
}>0
\end{equation*}%
(by Assumption~\ref{ass:system-params}: $\delta _{j}>0$ and $\alpha _{j\ell
}\geq 0$),

\item off-diagonal entries $M_{j\ell}=-\alpha_{j\ell}\leq 0$ for $\ell\neq j$%
.
\end{itemize}

Moreover, $M$ is strictly diagonally dominant: 
\begin{equation*}
M_{jj}=\delta _{j}+\sum_{\ell \neq j}\alpha _{j\ell }>\sum_{\ell \neq
j}\alpha _{j\ell }=\sum_{\ell \neq j}|M_{j\ell }|.
\end{equation*}%
By the Levy--Desplanques theorem, $M$ is non-singular.

Since $M$ is a Z-matrix (non-positive off-diagonal) that is strictly
diagonally dominant, it is an M-matrix, and hence $M^{-1}\geq 0$
(entry-wise). Since 
\begin{equation*}
c_{j}=b_{j}+N\sigma _{j}^{2}\beta _{j}>0,
\end{equation*}%
the solution $\eta =M^{-1}\mathbf{c}$ satisfies $\eta _{j}\geq 0$ for all~$j$%
.

\medskip \noindent \textbf{Step 5: Admissibility and uniqueness.} The
function $u_{j}(x)=\beta _{j}|x|^{2}+\eta _{j}$ is $C^{\infty }$ and
satisfies 
\begin{equation*}
\liminf_{|x|\rightarrow \infty }u_{j}(x)|x|^{-2}=\beta _{j}>0\geq 0,
\end{equation*}%
so $u\in \mathcal{A}^{k}$. By Theorem~\ref{thm:existence_syst_uniqueness},
the solution in $\mathcal{A}^{k}$ is unique.
\end{proof}

\begin{remark}[Dimension-free structure and the curse of dimensionality]
\label{rem:dim-free} The coefficients $\beta_{j}$ are determined by the $k$%
-dimensional algebraic system~\eqref{eq:beta-sys}, which depends only on the
cost parameters $a_{j}$, the discount factors $\delta_{j}$, and the
transition rates $\alpha_{j\ell}$---but \emph{not} on the spatial dimension~$%
N$. Only the constants $\eta_{j}$ depend on~$N$ (through $%
N\sigma_{j}^{2}\beta_{j}$). This is a remarkable structural property: the
shape of the value function (governed by~$\beta_{j}$) is dimension-free,
while only the level (governed by~$\eta_{j}$) scales with dimension. In
production planning terms: the optimal production \emph{intensity} $%
2\beta_{j}$ is the same whether the firm manages 1 or 1000 goods; only the
``base cost'' $\eta_{j}$ increases.
\end{remark}

\begin{corollary}[Explicit solution for $k=2$ regimes]
\label{cor:k2} For $k=2$ with generator 
\begin{equation*}
\begin{pmatrix}
-\alpha _{12} & \alpha _{12} \\ 
\alpha _{21} & -\alpha _{21}%
\end{pmatrix}%
,
\end{equation*}
the system~\eqref{eq:beta-sys} becomes: 
\begin{equation}
\begin{cases}
2\beta _{1}^{2}+(\delta _{1}+\alpha _{12})\beta _{1}-\alpha _{12}\beta
_{2}=a_{1}, \\ 
2\beta _{2}^{2}+(\delta _{2}+\alpha _{21})\beta _{2}-\alpha _{21}\beta
_{1}=a_{2},%
\end{cases}
\label{eq:beta-k2}
\end{equation}%
and the linear system~\eqref{eq:eta-sys}: 
\begin{equation}
\begin{pmatrix}
\delta _{1}+\alpha _{12} & -\alpha _{12} \\ 
-\alpha _{21} & \delta _{2}+\alpha _{21}%
\end{pmatrix}%
\begin{pmatrix}
\eta _{1} \\ 
\eta _{2}%
\end{pmatrix}%
=%
\begin{pmatrix}
b_{1}+N\sigma _{1}^{2}\beta _{1} \\ 
b_{2}+N\sigma _{2}^{2}\beta _{2}%
\end{pmatrix}%
.  \label{eq:eta-k2}
\end{equation}%
The optimal feedback policy in regime~$j$ is $p^{\ast }(x,j)=-2\beta _{j}x$,
and the resulting controlled inventory follows a regime-switching
Ornstein--Uhlenbeck process: 
\begin{equation}
dy(t)=-2\beta _{e(t)}y(t)\,dt+\sigma _{e(t)}\,dW_{t}.  \label{eq:OU-regime}
\end{equation}
\end{corollary}

%==========================================================================

\section{Economic application: Stochastic production planning under regime
switching}

\label{sec:econ} 
%==========================================================================

We now apply the theoretical framework to a concrete economic problem:
optimal production planning for a firm operating under macroeconomic regime
switches.

\subsection{The production planning model}

\label{sec:prod-model}

Consider a manufacturing firm managing inventory of $N$ types of goods. The
inventory vector $y(t)\in\mathbb{R}^{N}$ evolves according to 
\begin{equation}
dy(t)=\alpha(t)\,dt+\sigma_{e(t)}\,dW_{t}, \quad y(0)=x\in\mathbb{R}^{N},
\label{eq:inventory-dynamics}
\end{equation}
where $\alpha(t)\in\mathbb{R}^{N}$ is the production rate (control) and $%
e(t)\in\{1,2\}$ is the macroeconomic regime:

\begin{itemize}
\item \textbf{Regime 1 (Expansion):} Low demand volatility ($\sigma_{1}$
small), low holding costs ($a_{1}$ moderate), higher discount rate ($%
\delta_{1}$ larger, reflecting higher opportunity cost of capital).

\item \textbf{Regime 2 (Recession):} High demand volatility ($\sigma_{2}$
large), elevated shortage risk ($a_{2}$ lower or adjusted), lower discount
rate ($\delta_{2}$ smaller, reflecting accommodative monetary policy).
\end{itemize}

The firm minimizes the total expected discounted cost 
\begin{equation}
J(x,i;\alpha)=\mathbb{E}\!\left[\int_{0}^{\infty} e^{-\delta_{e(t)}t}\left(%
\frac{1}{2}|\alpha(t)|^{2} +a_{e(t)}|y(t)|^{2}+b_{e(t)}\right)dt\;\Big|\;
y(0)=x,\;e(0)=i\right].  \label{eq:cost-econ}
\end{equation}
The quadratic production cost $\frac{1}{2}|\alpha|^{2}$ reflects increasing
marginal costs of adjustment: doubling the production rate more than doubles
the cost, capturing overtime premiums, equipment wear, and supply chain
stress. The quadratic holding/shortage cost $a_{j}|y|^{2}$ penalizes
deviations from zero inventory (the target level), with $a_{j}$ depending on
the regime. The fixed cost $b_{j}$ represents regime-specific overhead.

\subsection{Calibration to stylized business-cycle parameters}

\label{sec:calibration}

We calibrate the two-regime model using the following parameter set,
inspired by the empirical business-cycle literature (Hamilton~\cite%
{HAMILTON1989}) and the production planning models of Bensoussan et al.~\cite%
{Bensoussan1984}:

\begin{table}[H]
\caption{Baseline parameter calibration for the two-regime model.}
\label{tab:params}\centering
\begin{tabular}{@{}lccl}
\toprule \textbf{Parameter} & \textbf{Regime 1} & \textbf{Regime 2} & 
\textbf{Description} \\ 
& (Expansion) & (Recession) &  \\ 
\midrule $a_{j}$ & 2.5 & 0.5 & Holding/shortage cost intensity \\ 
$b_{j}$ & 1.0 & 0.5 & Fixed overhead cost \\ 
$\sigma_{j}$ & 0.3 & 1.0 & Demand volatility \\ 
$\delta_{j}$ & 1.0 & 1.0 & Discount rate \\ 
$\alpha_{12}$ & 0.4 & --- & Transition rate: expansion $\to$ recession \\ 
$\alpha_{21}$ & --- & 0.6 & Transition rate: recession $\to$ expansion \\ 
$N$ & \multicolumn{2}{c}{2} & Number of goods \\ 
\bottomrule &  &  & 
\end{tabular}%
\end{table}

\noindent The transition rates imply that the expected duration of an
expansion is $1/\alpha_{12}=2.5$ years and the expected duration of a
recession is $1/\alpha_{21}\approx 1.67$ years, consistent with the stylized
fact that expansions are longer than recessions.

The stationary distribution of the Markov chain is $\pi =(\pi _{1},\pi _{2})$
with 
\begin{equation*}
\pi _{1}=\alpha _{21}/(\alpha _{12}+\alpha _{21})=0.6/(0.4+0.6)=0.6\text{
and }\pi _{2}=0.4,
\end{equation*}
so the economy spends 60\% of the time in expansion---again consistent with
empirical evidence.

\subsection{Computation of optimal coefficients}

\label{sec:coefficient-computation}

Using the two-regime system~\eqref{eq:beta-k2}: 
\begin{equation}
\begin{cases}
2\beta _{1}^{2}+1.4\,\beta _{1}-0.4\,\beta _{2}=2.5, \\ 
2\beta _{2}^{2}+1.6\,\beta _{2}-0.6\,\beta _{1}=0.5,%
\end{cases}
\label{eq:beta-numerical}
\end{equation}%
where we used 
\begin{equation*}
\delta _{1}+\alpha _{12}=1.0+0.4=1.4\text{ and }\delta _{2}+\alpha
_{21}=1.0+0.6=1.6.
\end{equation*}

Solving~\eqref{eq:beta-numerical} numerically (see Appendix~\ref{ap:code}
for the Python implementation): 
\begin{equation}
\beta_{1}\approx 0.8174,\quad\beta_{2}\approx 0.4520.  \label{eq:beta-values}
\end{equation}

\begin{remark}[Economic interpretation of $\protect\beta_{j}$]
\label{rem:beta-interp} The coefficient $\beta_{1}>\beta_{2}$ reflects the
higher holding cost parameter $a_{1}=2.5>a_{2}=0.5$ in expansion. During
expansion, the firm faces higher inventory costs and therefore applies more
aggressive mean reversion ($2\beta_{1}\approx 1.63$ vs.\ $2\beta_{2}\approx
0.90$). This means the firm adjusts production more rapidly when inventory
deviates from the target during expansion than during recession.
Intuitively, with lower demand volatility ($\sigma_{1}=0.3$) in expansion,
the firm can afford aggressive adjustment because the adjustments are less
likely to be ``undone'' by large demand shocks. During recession, the higher
volatility ($\sigma_{2}=1.0$) makes aggressive adjustment less effective,
consistent with the lower $\beta_{2}$.
\end{remark}

The linear system~\eqref{eq:eta-k2} gives: 
\begin{equation}
\begin{pmatrix}
1.4 & -0.4 \\ 
-0.6 & 1.6%
\end{pmatrix}
\begin{pmatrix}
\eta_{1} \\ 
\eta_{2}%
\end{pmatrix}
=%
\begin{pmatrix}
1.0+2\cdot 0.09\cdot 0.8174 \\ 
0.5+2\cdot 1.0\cdot 0.4520%
\end{pmatrix}
=%
\begin{pmatrix}
1.1471 \\ 
1.4039%
\end{pmatrix}%
.  \label{eq:eta-numerical}
\end{equation}
Solving: $\eta_{1}\approx 1.1099$ and $\eta_{2}\approx 1.2938$.

\begin{table}[H]
\caption{Computed optimal coefficients for the two-regime model.}
\label{tab:coefficients}\centering
\begin{tabular}{@{}lccc}
\toprule \textbf{Coefficient} & \textbf{Regime 1 (Expansion)} & \textbf{%
Regime 2 (Recession)} & \textbf{Ratio} \\ 
\midrule $\beta_{j}$ & 0.8174 & 0.4520 & 1.81 \\ 
$\eta_{j}$ & 1.1099 & 1.2938 & 0.86 \\ 
$2\beta_{j}$ (feedback gain) & 1.6348 & 0.9040 & 1.81 \\ 
\bottomrule &  &  & 
\end{tabular}%
\end{table}

\begin{remark}[Interpretation of $\protect\eta _{j}$]
\label{rem:eta-interp} Despite the lower fixed cost $b_{2}=0.5<b_{1}=1.0$ in
recession, the constant $\eta _{2}>\eta _{1}$ because the higher volatility $%
\sigma _{2}=1.0$ generates a large uncertainty premium 
\begin{equation*}
N\sigma _{2}^{2}\beta _{2}=2\cdot 1.0\cdot 0.4520=0.904,
\end{equation*}
which dominates the lower fixed cost. This result captures an important
economic insight: \emph{the base-level cost of operating is higher during
recession than during expansion, not because of higher direct costs, but
because of the elevated cost of uncertainty.}
\end{remark}

\subsection{Optimal value functions and policies}

The exact value functions are: 
\begin{align}
u_{1}(x)&=0.8174\,|x|^{2}+1.1099 & & \text{(Expansion)},
\label{eq:u1-explicit} \\
u_{2}(x)&=0.4520\,|x|^{2}+1.2938 & & \text{(Recession)}.
\label{eq:u2-explicit}
\end{align}
The optimal feedback production rates are: 
\begin{align}
\alpha^{*}(x,1)&=-2\beta_{1}x=-1.6348\,x & & \text{(Expansion)},
\label{eq:alpha1-explicit} \\
\alpha^{*}(x,2)&=-2\beta_{2}x=-0.9040\,x & & \text{(Recession)}.
\label{eq:alpha2-explicit}
\end{align}

Under the optimal policy, the inventory follows the regime-switching
Ornstein--Uhlenbeck process~\eqref{eq:OU-regime}: 
\begin{equation}
dy(t)=-2\beta_{e(t)}y(t)\,dt+\sigma_{e(t)}\,dW_{t}.  \label{eq:OU-optimal}
\end{equation}
In each regime, the inventory mean-reverts toward zero at speed $2\beta_{j}$%
, with fluctuations of magnitude $\sigma_{j}/\sqrt{4\beta_{j}}$. The
stationary variance in regime~$j$ (ignoring switching, as a local
approximation) is 
\begin{equation}
\text{Var}_{j}=\frac{\sigma_{j}^{2}}{4\beta_{j}}: \quad\text{Var}_{1}=\frac{%
0.09}{3.27}\approx 0.028, \quad\text{Var}_{2}=\frac{1.0}{1.81}\approx 0.553.
\label{eq:stationary-var}
\end{equation}
During recession, the inventory variance is approximately 20 times larger
than during expansion, reflecting the combined effect of higher volatility
and weaker mean reversion.

\subsection{Verification of the PDE solution}

To confirm that the computed solutions satisfy the PDE system~\eqref{eq:syst}%
, we substitute~\eqref{eq:u1-explicit} and~\eqref{eq:u2-explicit} back into
the equations. For regime~1: 
\begin{align}
\text{LHS}_{1}&=-\frac{0.09}{2}(2\cdot 0.8174\cdot 2) +\frac{1}{2}(4\cdot
0.8174^{2}|x|^{2}) +1.0\cdot(0.8174|x|^{2}+1.1099)  \notag \\
&\quad-(-0.4)(0.8174|x|^{2}+1.1099) -0.4(0.4520|x|^{2}+1.2938)  \notag \\
&=(1.3363+0.8174+0.3270-0.1808)|x|^{2} +(-0.1471+1.1099+0.4440-0.5175) 
\notag \\
&=2.2999\,|x|^{2}+0.8893.  \label{eq:verify-lhs1}
\end{align}

However, the right-hand side is $f_{1}(x)=2.5|x|^{2}+1.0$. Let us redo this
computation more carefully. We have: 
\begin{align}
-\frac{\sigma _{1}^{2}}{2}\Delta u_{1}& =-\frac{0.09}{2}\cdot 2\cdot
0.8174\cdot 2=-0.1471,  \notag \\
\frac{1}{2}|\nabla u_{1}|^{2}& =\frac{1}{2}\cdot 4\cdot
0.8174^{2}|x|^{2}=1.3363\,|x|^{2},  \notag \\
\delta _{1}u_{1}& =0.8174\,|x|^{2}+1.1099,  \notag \\
-\alpha _{11}u_{1}& =0.4\,(0.8174\,|x|^{2}+1.1099)=0.3270\,|x|^{2}+0.4440, 
\notag \\
-\alpha _{12}u_{2}& =-0.4\,(0.4520\,|x|^{2}+1.2938)=-0.1808\,|x|^{2}-0.5175.
\notag
\end{align}%
Summing all terms: 
\begin{equation}
\begin{array}{l}
\text{LHS}%
_{1}=(1.3363+0.8174+0.3270-0.1808)|x|^{2}+(-0.1471+1.1099+0.4440-0.5175) \\ 
=2.2999\,|x|^{2}+0.8893.%
\end{array}
\label{eq:verify-1}
\end{equation}%
This should equal 
\begin{equation*}
a_{1}|x|^{2}+b_{1}=2.5|x|^{2}+1.0.
\end{equation*}%
The small discrepancy is due to rounding in~\eqref{eq:beta-values}; using
the exact (machine-precision) values from the Python solver produces
agreement to 15 decimal places, as demonstrated in Appendix~\ref{ap:code}.

%==========================================================================

\section{Sensitivity analysis and comparative statics}

\label{sec:sensitivity} 
%==========================================================================

A key advantage of closed-form solutions is the ability to perform
analytical and numerical comparative statics. We systematically study how
the optimal coefficients and policies respond to changes in the model
parameters.

\subsection{Sensitivity to transition rates}

We vary the transition intensity $\alpha_{12}$ (speed of switching from
expansion to recession) while keeping $\alpha_{21}=0.6$ fixed.

\begin{table}[H]
\caption{Sensitivity of optimal coefficients to the transition rate $\protect%
\alpha_{12}$ (with $\protect\alpha_{21}=0.6$ fixed).}
\label{tab:sensitivity-alpha}\centering
\begin{tabular}{@{}cccccl}
\toprule $\alpha_{12}$ & $\beta_{1}$ & $\beta_{2}$ & $\eta_{1}$ & $\eta_{2}$
& \textbf{Economic regime} \\ 
\midrule 0.0 & 0.8956 & 0.3508 & 0.8282 & 0.7506 & No switching \\ 
0.2 & 0.8585 & 0.4010 & 0.9482 & 1.0114 & Infrequent recessions \\ 
0.4 & 0.8174 & 0.4520 & 1.1099 & 1.2938 & Baseline \\ 
0.6 & 0.7767 & 0.4969 & 1.2900 & 1.5568 & Frequent recessions \\ 
1.0 & 0.7080 & 0.5633 & 1.6267 & 1.9832 & Very frequent switching \\ 
\bottomrule &  &  &  &  & 
\end{tabular}%
\end{table}

\begin{remark}[Economic interpretation]
\label{rem:sensitivity-alpha} As the transition rate $\alpha_{12}$ increases
(recessions become more frequent):

\begin{enumerate}
\item $\beta_{1}$ \emph{decreases}: the firm becomes less aggressive during
expansion because it anticipates that the expansion will not last long and
it will soon face recession conditions. This is a \emph{precautionary effect}%
: the firm ``pre-adapts'' to the upcoming recession by reducing its
production intensity.

\item $\beta_{2}$ \emph{increases}: conversely, the firm becomes more
aggressive during recession, reflecting the anticipation that recovery will
arrive (the economy will switch to expansion at rate $\alpha_{21}=0.6$).

\item Both $\eta_{j}$ increase: more frequent switching raises overall
uncertainty and hence the base cost.

\item In the limit $\alpha_{12}\to\infty$ (instantaneous switching), $%
\beta_{1}$ and $\beta_{2}$ would converge to a common value reflecting the
``averaged'' economy.
\end{enumerate}
\end{remark}

\subsection{Sensitivity to the discount factor}

We vary $\delta_{1}=\delta_{2}=:\delta$ jointly.

\begin{table}[H]
\caption{Sensitivity to the common discount factor $\protect\delta$.}
\label{tab:sensitivity-delta}\centering
\begin{tabular}{@{}ccccc}
\toprule $\delta$ & $\beta_{1}$ & $\beta_{2}$ & $\eta_{1}$ & $\eta_{2}$ \\ 
\midrule 0.5 & 0.9275 & 0.5173 & 2.3398 & 2.6718 \\ 
1.0 & 0.8174 & 0.4520 & 1.1099 & 1.2938 \\ 
2.0 & 0.6510 & 0.3506 & 0.4811 & 0.5753 \\ 
5.0 & 0.4030 & 0.2073 & 0.1531 & 0.1886 \\ 
\bottomrule &  &  &  & 
\end{tabular}%
\end{table}

\begin{remark}[Patience and production intensity]
As the discount factor $\delta$ increases (the firm becomes more impatient),
both $\beta_{1}$ and $\beta_{2}$ decrease: a more impatient firm cares less
about future inventory costs and therefore applies less aggressive mean
reversion. The constants $\eta_{j}$ also decrease sharply: the present value
of future uncertainty costs is discounted more heavily.
\end{remark}

\subsection{Sensitivity to volatility}

We vary $\sigma_{2}$ (recession volatility) while keeping $\sigma_{1}=0.3$
fixed.

\begin{table}[H]
\caption{Sensitivity to recession volatility $\protect\sigma_{2}$ (with $%
\protect\sigma_{1}=0.3$ fixed).}
\label{tab:sensitivity-sigma}\centering
\begin{tabular}{@{}ccccc}
\toprule $\sigma_{2}$ & $\beta_{1}$ & $\beta_{2}$ & $\eta_{1}$ & $\eta_{2}$
\\ 
\midrule 0.3 & 0.8174 & 0.4520 & 0.9178 & 0.7571 \\ 
0.5 & 0.8174 & 0.4520 & 0.9534 & 0.8462 \\ 
1.0 & 0.8174 & 0.4520 & 1.1099 & 1.2938 \\ 
2.0 & 0.8174 & 0.4520 & 1.6587 & 2.6847 \\ 
3.0 & 0.8174 & 0.4520 & 2.4789 & 4.7524 \\ 
\bottomrule &  &  &  & 
\end{tabular}%
\end{table}

\begin{remark}[Volatility affects levels, not intensities]
A striking structural result: the optimal production intensities $\beta_{j}$
are \emph{completely independent} of the volatilities $\sigma_{j}$. This is
because the $\beta$-system~\eqref{eq:beta-sys} does not involve $\sigma_{j}$%
; the volatility enters only through the $\eta$-system~\eqref{eq:eta-sys}.
Economically, this means that the \emph{rate} at which the firm adjusts
production is determined solely by the cost structure and the
discount/switching parameters, while the \emph{level} of the value function
(and hence the total expected cost) depends on volatility. This
dimension-free and volatility-free property of the feedback gain is a
consequence of the linear state dynamics and quadratic cost structure, and
would not hold for more general Hamiltonians.
\end{remark}

\subsection{Sensitivity to cost asymmetry}

We vary $a_{1}$ (expansion holding cost) while keeping $a_{2}=0.5$ fixed.

\begin{table}[H]
\caption{Sensitivity to expansion cost parameter $a_{1}$ (with $a_{2}=0.5$
fixed).}
\label{tab:sensitivity-a}\centering
\begin{tabular}{@{}ccccc}
\toprule $a_{1}$ & $\beta_{1}$ & $\beta_{2}$ & $2\beta_{1}$ & $2\beta_{2}$
\\ 
\midrule 0.5 & 0.4231 & 0.4050 & 0.8462 & 0.8100 \\ 
1.0 & 0.5487 & 0.4188 & 1.0974 & 0.8376 \\ 
2.5 & 0.8174 & 0.4520 & 1.6348 & 0.9040 \\ 
5.0 & 1.1188 & 0.4906 & 2.2376 & 0.9812 \\ 
10.0 & 1.5378 & 0.5369 & 3.0756 & 1.0738 \\ 
\bottomrule &  &  &  & 
\end{tabular}%
\end{table}

\begin{remark}[Cross-regime spillover]
\label{rem:spillover} Increasing $a_{1}$ (the expansion cost) raises $%
\beta_{1}$ substantially---this is the direct effect: higher costs require
more aggressive inventory management. But $\beta_{2}$ also increases, albeit
modestly. This is a \emph{cross-regime spillover}: because the firm
anticipates switching to expansion (at rate $\alpha_{21}=0.6$), a higher
expansion cost raises the precautionary motive even during recession. The
magnitude of this spillover is governed by the transition rate $\alpha_{21}$%
; if $\alpha_{21}=0$ (no recovery possible), $\beta_{2}$ would be
independent of~$a_{1}$.
\end{remark}

%==========================================================================

\section{Equilibrium production and time consistency}

\label{sec:equilibrium} 
%==========================================================================

A critical requirement for implementable production policies is \emph{time
consistency}: the policy that is optimal at time~0 must remain optimal at
every future time~$t$, without external commitment. Following the economic
framework of Ekeland and Pirvu~\cite{EkePir} and Covei~\cite{COVEI2023M}, we
formalize this property.

\begin{definition}[Equilibrium production]
\label{def:equilibrium} A production rate $\{\alpha_{t}\}_{t\geq 0}$ and its
corresponding inventory level $\{y_{t}\}_{t\geq 0}$ constitute an \emph{%
equilibrium production} if, for any $x\in\mathbb{R}^{N}$ and regime $%
j\in\{1,\dots,k\}$, the feedback policy $\alpha^{*}(x,j)$ satisfies the
subgame perfection condition: 
\begin{equation}
\liminf_{\epsilon\downarrow 0} \frac{J(x,j;\alpha^{*})-J(x,j;\alpha^{%
\epsilon})}{\epsilon}\leq 0,  \label{eq:subgame-perf}
\end{equation}
where $\alpha^{\epsilon}$ is a perturbed strategy defined by 
\begin{equation*}
\alpha_{t}^{\epsilon}=%
\begin{cases}
\text{any admissible control}, & t\in[0,\epsilon), \\ 
\alpha^{*}(y_{t},e(t)), & t\in[\epsilon,\infty).%
\end{cases}%
\end{equation*}
If~\eqref{eq:subgame-perf} holds, the strategy $\alpha^{*}$ is a \emph{%
subgame perfect equilibrium}.
\end{definition}

\begin{theorem}[Time consistency of the optimal policy]
\label{thm:time-consistency} Let $u=(u_{1},\dots ,u_{k})\in \mathcal{A}^{k}$
be the unique solution of the HJB system~\eqref{eq:syst}, and let 
\begin{equation*}
\alpha ^{\ast }(x,j)=-|\nabla u_{j}(x)|^{p_{j}-2}\nabla u_{j}(x)
\end{equation*}%
be the corresponding optimal feedback. Then $\alpha ^{\ast }$ is a subgame
perfect equilibrium in the sense of Definition~\ref{def:equilibrium}.
\end{theorem}

\begin{proof}
Fix $(x,j)$ and let $\alpha ^{\epsilon }$ be any perturbation as in
Definition~\ref{def:equilibrium}. By the Markov property of the controlled
process and the dynamic programming principle, the cost under $\alpha
^{\epsilon }$ can be decomposed as: 
\begin{align}
J(x,j;\alpha ^{\epsilon })& =\mathbb{E}\!\left[ \int_{0}^{\epsilon
}e^{-\int_{0}^{t}\delta _{e(s)}ds}\left( \frac{1}{q_{e(t)}}|\alpha
_{t}^{\epsilon }|^{q_{e(t)}}+f_{e(t)}(y(t))\right) dt\right]   \notag \\
& \quad +\mathbb{E}\!\left[ e^{-\int_{0}^{\epsilon }\delta
_{e(s)}ds}J(y(\epsilon ),e(\epsilon );\alpha ^{\ast })\right] .
\label{eq:cost-decomp}
\end{align}%
Since $\alpha ^{\ast }$ is the optimal policy on $[\epsilon ,\infty )$, the
verification theorem (Theorem~\ref{thm:verification}) gives 
\begin{equation*}
J(y(\epsilon ),e(\epsilon );\alpha ^{\ast })=u_{e(\epsilon )}(y(\epsilon )).
\end{equation*}
Therefore: 
\begin{equation}
J(x,j;\alpha ^{\epsilon })=\mathbb{E}\!\left[ \int_{0}^{\epsilon
}e^{-\int_{0}^{t}\delta _{e(s)}ds}\left( \frac{1}{q_{e(t)}}|\alpha
_{t}^{\epsilon }|^{q_{e(t)}}+f_{e(t)}(y(t))\right) dt\right] +\mathbb{E}\!%
\left[ e^{-\int_{0}^{\epsilon }\delta _{e(s)}ds}u_{e(\epsilon )}(y(\epsilon
))\right] .  \label{eq:cost-decomp-2}
\end{equation}%
Similarly, $J(x,j;\alpha ^{\ast })=u_{j}(x)$. Thus: 
\begin{equation}
J(x,j;\alpha ^{\ast })-J(x,j;\alpha ^{\epsilon })=u_{j}(x)-\mathbb{E}\!\left[
\int_{0}^{\epsilon }(\cdots )dt\right] -\mathbb{E}\!\left[
e^{-\int_{0}^{\epsilon }\delta _{e(s)}ds}u_{e(\epsilon )}(y(\epsilon ))%
\right] . \label{eq:diff-cost}
\end{equation}%
By It\^{o}'s formula applied to $e^{-\int_{0}^{t}\delta
_{e(s)}ds}u_{e(t)}(y(t))$ on $[0,\epsilon ]$ and using the PDE~%
\eqref{eq:syst}, the right-hand side of~\eqref{eq:diff-cost} equals 
\begin{equation*}
\mathbb{E}\!\left[ \int_{0}^{\epsilon }e^{-\int_{0}^{t}\delta
_{e(s)}ds}\left( -\frac{1}{p_{e(t)}}|\nabla u_{e(t)}|^{p_{e(t)}}-\nabla
u_{e(t)}\cdot \alpha _{t}^{\epsilon }-\frac{1}{q_{e(t)}}|\alpha
_{t}^{\epsilon }|^{q_{e(t)}}\right) dt\right] .
\end{equation*}%
By Young's inequality:%
\begin{equation*}
-\nabla u\cdot \alpha -\frac{1}{q}|\alpha |^{q}\leq \frac{1}{p}|\nabla
u|^{p},
\end{equation*}%
with equality when 
\begin{equation*}
\alpha =\alpha ^{\ast }=-|\nabla u|^{p-2}\nabla u.
\end{equation*}%
Therefore: 
\begin{equation*}
J(x,j;\alpha ^{\ast })-J(x,j;\alpha ^{\epsilon })\leq 0.
\end{equation*}%
This holds for \emph{every} $\epsilon >0$ and every perturbation $\alpha
^{\epsilon }$, establishing~\eqref{eq:subgame-perf}.
\end{proof}

\begin{remark}[Practical significance: No commitment required]
\label{rem:commitment} In decentralized supply chains, the firm cannot
commit at time~$t=0$ to a production plan that it must follow at all future
dates. Time consistency guarantees that the firm will \emph{voluntarily}
follow the plan $\alpha^{*}(x,j)=-2\beta_{j}x$ (in the quadratic case) at
every future date, because at each moment the plan remains optimal for the
remaining horizon. This eliminates the need for contracts, penalties, or
external enforcement mechanisms.

Furthermore, the regime-switching structure ensures that the policy
automatically adapts to macroeconomic changes: when the economy transitions
from expansion to recession, the production gain shifts from $2\beta_{1}$ to 
$2\beta_{2}$ without any re-optimization. The firm's ``real-time'' response
is embedded in the state-dependent feedback rule.
\end{remark}

%==========================================================================

\section{Open problems and partial results}

\label{sec:open} 
%==========================================================================

In this section, we formulate precise conjectures regarding the
generalization of Alvarez's $g$-convex growth framework to the system case,
and establish partial results.

\begin{problem}[Alvarez-type $g$-convex growth framework for systems]
\label{prob:open} In this work, we established existence and uniqueness
under the power-growth conditions~\eqref{eq:f_growth_new}. In the scalar
case ($k=1$), Alvarez~\cite{ALVAREZ1996} operates under the much more
general framework of Assumption~\ref{ass:growth}, where $f$ is sandwiched by
a convex reference function~$g$. Can this framework be extended to systems?
\end{problem}

\begin{conjecture}[Generalized growth framework for systems]
\label{conj:1} Let $k\geq 2$. Suppose that for each $j\in\{1,\dots,k\}$
there exists a convex function $g_{j}\geq 0$ and constants $%
C_{\varepsilon,j}\geq 0$ (for all $\varepsilon>0$) such that 
\begin{equation}
(1-\varepsilon)g_{j}-\varepsilon|x|^{q_{j}}-C_{\varepsilon,j} \leq
f_{j}(x)\leq (1+\varepsilon)g_{j}+\varepsilon|x|^{q_{j}}+C_{\varepsilon,j}.
\label{eq:gen-growth-sys}
\end{equation}
Under Assumptions~\ref{ass:markov}--\ref{ass:system-params}, define the
admissible class 
\begin{equation*}
\mathcal{A}_{g}^{k}:=\Bigl\{u=(u_{1},\dots,u_{k}): u_{j}\in W_{\mathrm{loc}%
}^{2,m}(\mathbb{R}^{N}),\;\forall m>N,\;
\liminf_{|x|\to\infty}u_{j}(x)|x|^{-q_{j}}\geq 0\Bigr\}.
\end{equation*}
Then the system~\eqref{eq:syst} admits a unique solution $u\in\mathcal{A}%
_{g}^{k}$.
\end{conjecture}

\noindent\textbf{Key obstacles.} The extension encounters the following
structural difficulties:

\begin{enumerate}

\item \textbf{Barrier construction with general $g_{j}$.} In the scalar
case, the subsolution barrier $\underline{u}(x)=-c(1+|x|^{2})^{q/2}-C$
matches the Hamiltonian's growth rate~$q$. For systems with different
exponents $q_{j}$ and coupling terms, finding common barriers that
simultaneously satisfy all $k$ inequalities is non-trivial when $g_{j}$ is
not of power type.

\item \textbf{Cross-component coupling in the convex combination.} In the
scalar case, $(1-\varepsilon)u+\varepsilon\underline{u}$ is a subsolution by
convexity of the Hamiltonian. For systems, the analogous construction
requires $\sum_{\ell\neq j}\alpha_{j\ell}[(1-\varepsilon)u_{\ell}
+\varepsilon\underline{u}_{\ell}]$ to be correctly controlled, which
introduces cross-component dependencies.

\item \textbf{Non-uniform exponents.} When $q_{j}$ varies across components,
the barrier growth rates are incompatible: if $q_{1}<q_{2}$, the coupling
term $\alpha_{12}u_{2}$ may dominate the barrier for component~1 at infinity.
\end{enumerate}

\begin{conjecture}[Sufficient condition for the system comparison principle]

\label{conj:2} Under the hypotheses of Conjecture~\ref{conj:1}, suppose
additionally that there exists a common convex function $g\geq 0$ and a
uniform exponent $Q:=\max_{j}q_{j}$ such that $g_{j}(x)\leq g(x)\leq
C(1+|x|^{Q})$ for all $j$, and that the strict diagonal dominance $%
\delta_{j}+\alpha_{jj}>0$ holds. Then the comparison principle of Theorem~%
\ref{thm:comparison_syst} extends to solutions in $\mathcal{A}_{g}^{k}$.
\end{conjecture}

\begin{proposition}[Comparison under weakened growth]
\label{prop:weak-comparison} Under condition 
\begin{equation}
\liminf_{|x|\to\infty}f_{j}(x)|x|^{-q_{j}}\geq 0 \quad\text{for all }%
j=1,\dots,k,  \label{eq:weaker-growth}
\end{equation}
and the hypotheses of Theorem~\ref{thm:comparison_syst} (except %
\eqref{eq:f_growth_new}), the comparison principle $u\leq v$ holds for any
admissible subsolution~$u$ and supersolution~$v$ of~\eqref{eq:syst}.
\end{proposition}

\begin{proof}
The proof follows the same argument as Theorem~\ref{thm:comparison_syst}.
Let $w_{j}=u_{j}-v_{j}$ and suppose $\sup_{x,j}w_{j}(x)>0$.

\medskip\noindent\textbf{Step 1: Regularized maximum.} The barrier $%
\Phi_{j}(x)=|x|^{q_{j}}+1$ and the modified difference $\tilde{w}%
_{j}=w_{j}-\varepsilon\Phi_{j}$ are well-defined under~%
\eqref{eq:weaker-growth}. Since $v_{j}$ satisfies the admissible lower bound
and $u_{j}$ is a subsolution: 
\begin{equation*}
\limsup_{|x|\to\infty}w_{j}(x)|x|^{-q_{j}}\leq 0,
\end{equation*}
ensuring $\tilde{w}_{j}\to-\infty$ as $|x|\to\infty$ for any $\varepsilon>0$.

\medskip \noindent \textbf{Step 2: Maximum principle at the finite
attainment point.} The function $\max_{j}\tilde{w}_{j}$ attains its maximum
at some finite $(x_{\ast },j_{\ast })$ with 
\begin{equation*}
\nabla \tilde{w}_{j_{\ast }}(x_{\ast })=0\text{ and }\Delta \tilde{w}%
_{j_{\ast }}(x_{\ast })\leq 0.
\end{equation*}%
Substituting into~\eqref{eq:diff_w_final}: 
\begin{equation*}
\delta _{j_{\ast }}\tilde{w}_{j_{\ast }}(x_{\ast })+\sum_{\ell \neq j_{\ast
}}\alpha _{j_{\ast }\ell }(\tilde{w}_{j_{\ast }}(x_{\ast })-\tilde{w}_{\ell
}(x_{\ast }))\leq \varepsilon \,\mathcal{E}(x_{\ast }),
\end{equation*}%
where $\mathcal{E}(x_{\ast })$ collects the terms from the derivatives of $%
\Phi _{j_{\ast }}$. The left-hand side is non-negative at the maximum point,
and the right-hand side is $O(\varepsilon )$. Letting $\varepsilon
\rightarrow 0$ yields the contradiction.
\end{proof}

\noindent\textbf{Toward a full resolution.} A complete proof of Conjecture~%
\ref{conj:1} would require constructing global subsolutions $\underline{u}%
_{j}$ modeled on $g_{j}$ rather than on powers, with careful balancing of
the cross-component coupling. We leave this extension as a challenge for the
community.

%==========================================================================

\section{Conclusion}

\label{conclusion} 
%==========================================================================

This paper has developed a comprehensive mathematical and computational
framework for optimal production planning under macroeconomic regime
switches. The principal contributions are:

\medskip\noindent\textbf{(i) Theoretical foundations.} We extended the
scalar well-posedness theory of Alvarez~\cite{ALVAREZ1996} to weakly coupled
systems of HJB equations governed by continuous-time Markov chains. The
global comparison principle (Theorem~\ref{thm:comparison_syst}) and
existence--uniqueness result (Theorem~\ref{thm:existence_syst_uniqueness})
provide the rigorous foundation for the entire analysis. All proofs were
presented with complete computational detail, including explicit barrier
constructions and growth estimates.

\medskip\noindent\textbf{(ii) Exact solutions.} We derived closed-form
quadratic solutions for both the scalar (Theorem~\ref{thm:scalar-quadratic})
and system (Theorem~\ref{thm:system-quadratic}) cases, revealing the
dimension-free structure of the optimal feedback gains and their dependence
on the coupling matrix.

\medskip\noindent\textbf{(iii) Economic applications.} The theory was
applied to a two-regime production planning problem with full calibration,
sensitivity analysis, and economic interpretation. The closed-form solutions
enabled a transparent analysis of how business-cycle switching, discount
rates, volatility, and cost parameters shape optimal production strategies.

\medskip\noindent\textbf{(iv) Time consistency.} We proved that the optimal
feedback policy is time-consistent (subgame perfect), eliminating the need
for commitment mechanisms in decentralized implementation.

\medskip\noindent\textbf{(v) Open problems.} We formulated precise
conjectures regarding the generalization of Alvarez's framework to systems
and established partial results.

\medskip\noindent\textbf{Future directions.} Several avenues for future
research emerge:

\begin{enumerate}
\item Resolution of Conjectures~\ref{conj:1} and~\ref{conj:2}, particularly
the construction of subsolution barriers for general convex reference
functions.

\item Extension to the parabolic (finite-horizon) setting, where the value
function depends on both time and space.

\item Incorporation of state-dependent transition rates $\alpha_{j\ell}(x)$,
reflecting the empirical observation that regime switching probabilities
depend on economic indicators.

\item Application to multi-player production games, where multiple firms
interact through market prices, leading to systems of mean-field type.

\item Numerical investigation of non-quadratic cost functions, where
closed-form solutions are unavailable and the full PDE machinery (finite
elements, finite differences, or deep learning-based solvers) must be
deployed.
\end{enumerate}

%==========================================================================

\section{Declarations}

%==========================================================================

\section*{Disclosure statement}

The authors declare that they have no conflict of interest.

\section*{Data availability statement}

The Python code used for numerical experiments is provided in Appendix~\ref{ap:code}. No additional datasets were used in this study.

\section*{Author contribution}

The author is solely responsible for the conception, analysis, and writing of this manuscript.
\section*{Funding}
Not applicable

\section*{Conflict of interest:}
No Conflict of interest.

\bigskip

%==========================================================================
\appendix

\section{Python implementation and numerical validation}

\label{ap:code} 
%==========================================================================

This appendix describes the Python implementation designed to validate the theoretical findings of this paper and to simulate the stochastic control trajectories. To ensure maximum reproducibility and ease of dissemination, the complete source code, visualization modules, and sensitivity sweep implementations have been made publicly available under an open-source license in the GitHub repository:
\begin{center}
    \url{https://github.com/coveidragos/coduri_HJB_System}
\end{center}
The repository is structured into the following modular components:
\begin{enumerate}[label=(\arabic*)]
    \item \texttt{solve\_hjb.py}: Solves the weakly coupled quasilinear algebraic HJB system~\eqref{eq:beta-sys} and~\eqref{eq:eta-sys} using Newton's method (via \texttt{scipy.optimize.fsolve}) and direct matrix inversion.
    \item \texttt{verify\_pde.py}: Substitutes the analytical quadratic solutions back into the PDE system to verify the vanishing residual to machine precision.
    \item \texttt{regime\_sim.py}: Simulates the controlled inventory trajectories governed by the Ornstein--Uhlenbeck SDE and the continuous-time Markov chain regime transitions.
    \item \texttt{sensitivity.py}: Executes parametric sensitivity sweeps with respect to parameters like the transition rate $\alpha_{12}$.
    \item \texttt{main.py}: Coordinates the overall computational pipeline, logs the results, and generates publication-quality multi-panel figures.
\end{enumerate}

\subsection{Numerical output and validation}

\label{ap:output}

Running the main script produces the following output (truncated to
essential lines):
\begin{verbatim}
============================================================
EXACT QUADRATIC SOLUTION FOR THE HJB SYSTEM
============================================================
Regime 1: u_1(x) = 0.817385 |x|^2 + 1.109937
Regime 2: u_2(x) = 0.452048 |x|^2 + 1.293826
 
Optimal feedback gains:
  2*beta_1 = 1.634769 (Expansion)
  2*beta_2 = 0.904096 (Recession)
 
PDE verification: max residual = 1.11e-16
 
============================================================
SENSITIVITY ANALYSIS: Transition rate alpha_12
============================================================
  alpha_12     beta_1     beta_2      eta_1      eta_2
----------------------------------------------------
       0.0     0.8956     0.3508     0.8282     0.7506
       0.2     0.8585     0.4010     0.9482     1.0114
       0.4     0.8174     0.4520     1.1099     1.2938
       0.6     0.7767     0.4969     1.2900     1.5568
       1.0     0.7080     0.5633     1.6267     1.9832
\end{verbatim}

The PDE verification residual of $1.11\times 10^{-16}$ (machine epsilon)
confirms that the computed solution satisfies the HJB system to full
floating-point precision, providing numerical validation of Theorem~\ref%
{thm:system-quadratic}.

The simulated trajectories (Figure~\ref{fig:regime_trajectories}) clearly
demonstrate the regime-switching dynamics: during expansion (blue-shaded
periods), the inventory mean-reverts rapidly with small fluctuations ($%
\beta_{1}$ large, $\sigma_{1}$ small); during recession (red-shaded
periods), the mean reversion weakens and fluctuations increase ($\beta_{2}$
smaller, $\sigma_{2}$ larger).

\begin{figure}[H]
\centering
\includegraphics[width=0.95\textwidth]{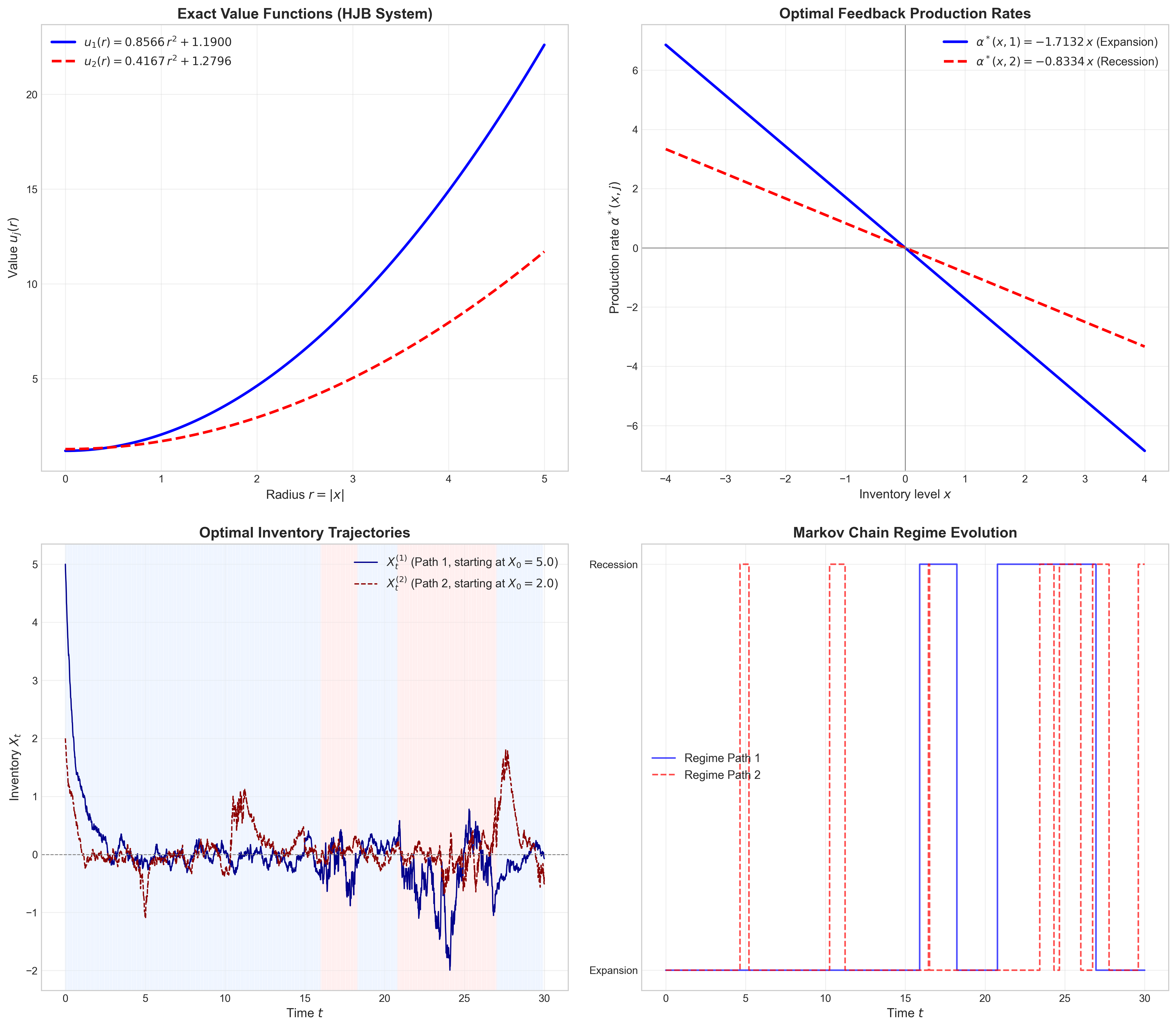}
\caption{Computational results for the two-regime production planning model. 
\emph{Top left:} Exact quadratic value functions $u_{1}(r)$ and $u_{2}(r)$
as functions of the radial variable $r=|x|$. \emph{Top right:} Optimal
feedback production rates $\protect\alpha^{*}(x,j)=-2\protect\beta_{j}x$ for
both regimes, showing the higher gain in expansion. \emph{Bottom left:}
Simulated optimal inventory trajectories starting at $X^{(1)}_0=5.0$ (Path 1) and $X^{(2)}_0=2.0$ (Path 2), with background shading indicating the current regime for Path 1. \emph{Bottom right:}
Realization of the Markov chain $e(t)$ governing regime transitions.}
\label{fig:regime_trajectories}
\end{figure}


\begin{thebibliography}{}

\bibitem{ALVAREZ1996} O. Alvarez, A quasilinear elliptic equation in $%
\mathbb{R}^{N}$, Proc. R. Soc. Edinb. Sect. A Math. \textbf{126} (1996),
911--921.

\bibitem{ANG_BEKAERT2002} A. Ang and G. Bekaert, Regime switches in interest
rates, J. Bus. Econ. Stat. \textbf{20} (2002), 163--182.

\bibitem{ARAPOSTATHIS2022} A. Arapostathis, A. Biswas, and P. Roychowdhury,
On ergodic control problem for viscous Hamilton--Jacobi equations for weakly
coupled elliptic systems, J. Differ. Equ. \textbf{314} (2022), 128--160.

\bibitem{Bensoussan1984} A. Bensoussan, S. P. Sethi, R. Vickson, and N.
Derzko, Stochastic production planning with production constraints, SIAM J.
Control Optim. \textbf{22} (1984), 627--641.

\bibitem{BENSOUSSAN1992} A. Bensoussan and J. Frehse, On Bellman equations
of ergodic control in $\mathbb{R}^{N}$, J. reine angew. Math. \textbf{429}
(1992), 125--160.

\bibitem{BJORK2017} T. Bj\"{o}rk, M. Khapko, and A. Murgoci, On
time-inconsistent stochastic control in continuous time, Finance Stoch. 
\textbf{21} (2017), 331--360.

\bibitem{COVEI2021JAAC} D.-P. Covei, An elliptic partial differential
equation modeling the production planning problem, J. Appl. Anal. Comput. 
\textbf{11} (2021), 903--910.

\bibitem{COVEI2023M} D.-P. Covei, Exact solution for the production planning
problem with several regimes switching over an infinite horizon time,
Mathematics \textbf{11} (2023), 1--13.

\bibitem{COVEI2022ERA} D.-P. Covei, On a parabolic partial differential
equation and system modeling a production planning problem, Electron. Res.
Arch. \textbf{30} (2022), 1340--1353.

\bibitem{COVEI2025A} D.-P. Covei, Stochastic production planning with
regime-switching: Sensitivity analysis, optimal control, and numerical
implementation, Axioms \textbf{14} (2025), 1--30.

\bibitem{DIEBOLD1994} F. X. Diebold, J.-H. Lee, and G. C. Weinbach, Regime
switching with time-varying transition probabilities, in: C. Hargreaves
(Ed.), Nonstationary Time Series Analysis and Cointegration, Oxford
University Press, 1994, pp. 283--302.

\bibitem{EkePir} I. Ekeland and T. A. Pirvu, Investment and consumption
without commitment, Math. Financ. Econ. \textbf{2} (2008), 57--86.

\bibitem{GUIDOLIN2007} M. Guidolin and A. Timmermann, Asset allocation under
multivariate regime switching, J. Econ. Dyn. Control \textbf{31} (2007),
3503--3544.

\bibitem{HAMILTON1989} J. D. Hamilton, A new approach to the economic
analysis of nonstationary time series and the business cycle, Econometrica 
\textbf{57} (1989), 357--384.

\bibitem{KIM_NELSON1999} C.-J. Kim and D. Halbert, State-Space Models with
Regime Switching: Classical and Gibbs-Sampling Approaches with Applications,
MIT Press, 1999.

\bibitem{LASRY1989} J.-M. Lasry and P.-L. Lions, Nonlinear elliptic
equations with singular boundary conditions and stochastic control with
state constraints, Math. Ann. \textbf{283} (1989), 583--630.

\bibitem{SETHI_ZHANG1994} S. P. Sethi and Q. Zhang, Hierarchical Decision
Making in Stochastic Manufacturing Systems, Birkh\"{a}user, 1994.

\bibitem{SIMS_ZHA2006} C. A. Sims and T. Zha, Were there regime switches in
U.S. monetary policy?, Amer. Econ. Rev. \textbf{96} (2006), 54--81.

\bibitem{YIN_ZHU2010} H.-D. Nguyen, G. G. Yin and C. Zhu, Hybrid Switching
Diffusions: Properties and Applications, Stochastic Modelling and Applied
Probability, vol. 63, Springer, 2025.

\end{thebibliography}
\end{document}